\documentclass{amsart}
\usepackage{amscd,amsmath,amssymb,amsfonts,bbm}
\usepackage[utf8]{inputenc}
\usepackage[T1]{fontenc}
\usepackage{lmodern}
\usepackage[unicode,pdfborder={0 0 0},final]{hyperref}
\usepackage{mathtools}
\usepackage{enumerate}
\usepackage[all,cmtip]{xy}
\usepackage{footmisc}

\newtheorem{thm}{Theorem}[section]
\newtheorem{prop}[thm]{Proposition}
\newtheorem{lem}[thm]{Lemma}
\newtheorem{cor}[thm]{Corollary}
\newtheorem{property}[thm]{Property}

\theoremstyle{definition}
\newtheorem{Def}[thm]{Definition}

\theoremstyle{remark}
\newtheorem{rem}[thm]{Remark}
\newtheorem{ex}[thm]{Example}

\numberwithin{equation}{section}

\newcommand{\bF}{\mathbb{F}}
\newcommand{\Q}{\mathbb{Q}}
\newcommand{\sH}{\mathcal{H}}
\newcommand{\sI}{\mathcal{I}}
\newcommand{\sO}{\mathcal{O}}
\newcommand{\sL}{\mathcal{L}}
\newcommand{\sE}{\mathcal{E}}
\newcommand{\et}{\mathrm{\acute{e}t}}
\DeclareMathOperator{\Br}{Br}
\DeclareMathOperator{\AJ}{AJ}
\newcommand{\Ab}{\mathrm{Ab}}
\newcommand{\alg}{\mathrm{alg}}

 \DeclareMathOperator{\Spec}{Spec}

\DeclareMathOperator{\Ker}{Ker}
\DeclareMathOperator{\Pic}{Pic}

\DeclareMathOperator{\Sym}{Sym}

\DeclareMathOperator{\Ima}{Im}
\newcommand{\Id}{\mathrm{Id}}

\newcommand{\ppav}{\mathrm{ppav}}
\DeclareMathOperator{\Aut}{Aut}
\newcommand{\tors}{\mathrm{tors}}

\newcommand{\bS}{\mathbb{S}}

\newcommand{\R}{\mathbb{R}}
\newcommand{\bP}{\mathbb{P}}
\newcommand{\bA}{\mathbb{A}}
\renewcommand{\C}{\mathbb{C}}
\DeclareMathOperator{\Hom}{Hom}
\DeclareMathOperator{\NS}{NS}

\DeclareMathOperator{\Gal}{Gal}
\DeclareMathOperator{\Res}{Res}
\newcommand{\Z}{\mathbb{Z}}
\newcommand{\nr}{\mathrm{nr}}
\DeclareMathOperator{\CH}{CH}

\newcommand{\ci}{\mathcal{C}^{\infty}}
\newcommand{\torsion}{\mathrm{torsion}}
\newcommand{\cl}{\mathrm{cl}}
\newcommand{\surj}{\twoheadrightarrow}
\newcommand{\inj}{\hookrightarrow}

\begin{document}

\date{March 19th, 2019; revised on December 18th, 2019}
\title[The Clemens--Griffiths method over non-closed fields]{The Clemens--Griffiths method over\\non-closed fields}

 \author{Olivier Benoist}
\address{D\'epartement de math\'ematiques et applications, \'Ecole normale sup\'erieure, 45~rue d'Ulm, 75230 Paris Cedex 05, France}
 \email{olivier.benoist@ens.fr}

 \author{Olivier Wittenberg}
 \address{Institut de Math\'ematique d'Orsay, B\^atiment 307, Universit\'e Paris-Sud, 91405 Orsay Cedex, France}
\email{olivier.wittenberg@math.u-psud.fr}

\renewcommand{\abstractname}{Abstract}
\begin{abstract}
We use the Clemens--Griffiths method to construct smooth projective threefolds, over any field $k$ admitting a separable quadratic extension, that are $k$-unirational and $\overline{k}$-rational but not $k$-rational.  When $k=\R$, we can moreover ensure that their real locus is diffeomorphic to the real locus of a smooth projective $\R$\nobreakdash-rational variety and that all their unramified cohomology groups are trivial.
\end{abstract}

\maketitle

\section{Introduction}\label{intro}

The L\"uroth problem aims at understanding when a variety $X$ over a field $k$ is $k$\nobreakdash-\textit{rational}, that is, birational to $\bP^n_k$. It is natural to restrict to classes of varieties that are close to being $k$-rational, such as $k$\nobreakdash-\textit{unirational} varieties, which admit a dominant rational map $\bP^n_k \dashrightarrow X$.

Over the field $k=\C$ of complex numbers, unirational surfaces are rational, and examples of non-rational unirational threefolds were discovered almost simultaneously by Artin--Mumford \cite{AM}, Clemens--Griffiths \cite{CG} and Iskovskikh--Manin~\cite{IM}. We refer to \cite{surveybeauville} for a beautiful survey of their methods and their rich legacy.
 
\vspace{1em}

Over a non-algebraically closed field~$k$ with algebraic closure~$\overline k$, it is interesting to investigate the $k$-rationality of varieties that are  $\overline{k}$-rational. Significant works in this direction include Chevalley's example of a torus over $\Q_p$ that is not $\Q_p$\nobreakdash-rational \cite[\S V]{Chevalley} and Swan's counter-example to Noether's problem over $\Q$ \cite[Theorem~1]{Swan}.

The strategies used by Iskovskikh and Manin (the Noether--Fano method of analyzing birational automorphism groups) and by Artin and Mumford (based on the study of Brauer groups)
have both been employed to construct interesting examples of $\overline{k}$-rational varieties that are not $k$-rational. Early applications to surfaces are respectively due to Segre (smooth cubic surfaces of Picard rank $1$ are never $k$\nobreakdash-rational \cite[Theorems 3 and 5]{Segre}, 
see also \cite[Theorem 2.1]{KSC}), and to Manin (see for instance \cite[Theorem 2.5]{Maninperfect}).

The main goal of this paper is to show that it is also possible to use the strategy of Clemens and Griffiths (relying on the theory of intermediate Jacobians) to construct varieties over~$k$ that are $\overline{k}$-rational but not $k$\nobreakdash-rational. Here are concrete new examples that we obtain in this way.

\begin{thm}[Corollary \ref{coroexemple} and Example \ref{extexte}]
\label{thexemple}
Let $k$ be a field of characteristic different from $2$. If $\alpha\in k^*\setminus (k^*)^2$, the $k$-variety defined by the affine equation $\{s^2-\alpha t^2=x^{4}+y^{4}+1\}$ is $k$-unirational, $k(\sqrt{\alpha})$-rational but not $k$-rational.
\end{thm}

\begin{thm}[Corollary \ref{coroexemple2} and Example \ref{extexte2}]
\label{thexemple2}
Let $k$ be a field of characteristic~$2$. Let $\alpha\in k$ and $\beta\in \overline{k}\setminus k$ be such that $\beta^2+\beta=\alpha$. The $k$\nobreakdash-variety with affine equation $\{s^2+st+\alpha t^2=x^3y+y^3+x\}$ is $k$-unirational, $k(\beta)$-rational but not $k$-rational.
\end{thm}

Constructions of intermediate Jacobians over other fields than the field $\C$ of complex numbers have been provided by Deligne \cite{Deligne}, Murre \cite{Murrecubic, Murre} and Achter, Casalaina-Martin and Vial \cite{ACMV1,ACMV}, in various degrees of generality. In Section \ref{secij}, 
building on these works and using in an essential way Bloch's Abel--Jacobi map~\cite{Bloch}, we associate with any smooth projective $\overline{k}$\nobreakdash-rational threefold $X$ over a perfect field $k$ a principally polarized abelian variety $J^3X$ over $k$ (our contribution being the construction of the principal polarization). We verify in Corollary \ref{CGk} that it gives rise to an obstruction to the $k$-rationality of $X$ generalizing the one considered by Clemens and Griffiths \cite{CG}: if $X$ is $k$-rational, then $J^3X$ is isomorphic to the Jacobian of a (possibly disconnected) smooth projective curve over~$k$.

Over algebraically closed fields, several techniques have been used to detect that an intermediate Jacobian is not a Jacobian: the geometry of its theta divisor \cite{CG}, its automorphism group \cite{Beauvilleaut}, or the zeta function of one of its specializations over a finite field \cite{MR}. To give examples of $\overline{k}$-rational varieties that are not $k$\nobreakdash-rational, we need a criterion of a more algebraic nature, which can distinguish between Jacobians of curves and their twists. Such a criterion is established in Proposition~\ref{Torelli} as a consequence of the Torelli theorem.
It is especially easy to apply when $X$ itself is a twist of a $k$-rational variety (see Proposition \ref{twisttwist}), as in Theorems~\ref{thexemple} and \ref{thexemple2}.

\vspace{1em}

Our results are of particular interest over the field $k=\R$ of real numbers, with Galois group $G:=\Gal(\C/\R)\simeq \Z/2\Z$. The real locus of an $\R$-rational smooth projective variety is non-empty and connected.
That this yields obstructions to $\R$\nobreakdash-rationality goes back to Comessatti (\cite[\S 5]{Comessatti}, see also \cite[Th\'eor\`eme 1.1]{CTcubiques}).

In dimension $\leq 2$, there are no further obstructions to the $\R$-rationality of a $\C$-rational variety. The case of curves is easy since a real conic with a real point is isomorphic to $\bP^1_{\R}$, and it is a theorem of Comessatti that a smooth projective $\C$-rational surface over $\R$ whose real locus is non-empty and connected is $\R$-rational (see \cite[pp.\ 54-55]{Comessatti}
or the modern proof of Silhol \cite[VI Corollary 6.5]{Silhol}).

In dimension $\geq 3$, all known examples of smooth projective $\C$-rational varieties over $\R$ that
are not $\R$\nobreakdash-rational rely on a real analogue of the Artin--Mumford invariant (the Brauer group) or on its higher degree generalizations given by unramified cohomology \cite{CToj,Peyre}. (The latter take into account the obstructions induced by the number of connected components of the real locus by \cite[Main Theorem]{CTP}.)
We give the first example of an irrational smooth projective $\C$-rational variety over~$\R$ that does not rely on the above-mentioned invariants, dashing any hope for a simple $\R$-rationality criterion for $\C$-rational varieties in dimension $\geq 3$.

\begin{thm}[Theorem \ref{premierexemple}]
\label{main}
There exists a smooth projective threefold $X$ over~$\R$ that is not $\R$\nobreakdash-rational, but that is $\C$-rational, $\R$-unirational, such that $X(\R)$ is diffeomorphic to $(\bP^1\times\bP^2)(\R)$, and such that for any $G$-module $M$ and any $i\geq 0$,  $H^i(G,M)\stackrel{\sim}\longrightarrow H^i_{\nr}(X,M)$.
\end{thm}

The variety $X$ used in our proof of Theorem \ref{main} is the one described in Theorem~\ref{thexemple} for $\alpha=-1$. It is its intermediate Jacobian that shows that it is not $\R$\nobreakdash-rational. The last statement of Theorem \ref{main} asserts that the unramified cohomology groups of $X$ cannot be used to show that $X$ is not $\R$\nobreakdash-rational.  To contrast with Theorem \ref{main}, we provide in Theorem \ref{secondexemple} an example of a smooth projective $\C$\nobreakdash-rational and $\R$\nobreakdash-unirational threefold over $\R$ whose real locus is diffeomorphic to the real locus of a smooth projective $\R$-rational variety, whose intermediate Jacobian is trivial, but that
is not $\R$-rational thanks to the Artin--Mumford invariant.

A new specialization technique introduced by Voisin \cite{Voisinspe} has recently led to tremendous progress in rationality problems (see \cite{PeyreBourbaki} for a survey of this method and its applications). However, these specialization arguments cannot provide examples of $\C$-rational varieties over $\R$ that are not $\R$-rational, for the reason that all non-trivial valuations on $\R$ have algebraically closed residue fields. In particular, such arguments cannot be used to prove Theorem \ref{main}.

\vspace{1em}

Let us now explain and develop the last statement of Theorem \ref{main}, which concerns unramified cohomology.
 One can associate with any smooth projective variety $X$ over~$\R$ an abelian group $H^i_{\nr}(X,M)$ for every integer $i\geq 0$ and every $G$-module~$M$ (see \S\ref{unramcoho}). If $X$ is $\R$-rational, these unramified cohomology groups are trivial in the sense that the natural pull-back maps $H^i(G,M)\to H^i_{\nr}(X,M)$ are isomorphisms (Proposition \ref{trivialunramcoho}). 
  Relying on Bloch--Ogus theory, we study these invariants in Section \ref{secunram}.
Our main contribution is a complete understanding of when they can be used to show that a $\C$-rational threefold is not $\R$-rational, yielding a proof of the last assertion of Theorem \ref{main}.

\begin{thm}[Theorem \ref{thunr3}]
\label{thunr3intro}
Let $X$ be a smooth projective threefold over $\R$ that is $\C$-rational. The following are equivalent:
\begin{enumerate}
\item For any $i\geq 0$ and any $G$-module $M$, $H^i(G,M)\stackrel{\sim}\longrightarrow H^i_{\nr}(X,M)$.
\item The variety $X$ satisfies:
\begin{enumerate}[(i)]
\item $X(\R)$ has exactly one connected component,
\item $\Pic(X_{\C})$ is a permutation $G$-module,
\item The cycle class map $\cl_{\R}:\CH_1(X)\to H_1(X(\R),\Z/2\Z)$ is surjective.
\end{enumerate}
\end{enumerate}
\end{thm}

We have already discussed condition (i) in Theorem \ref{thunr3intro}. Manin \cite[Theorem~2.2]{Maninperfect} and Voskresenskii \cite[Theorem 1]{Voskr} noticed that there are restrictions on the Galois module structure of the geometric Picard group of smooth projective $k$\nobreakdash-rational varieties. When $k=\R$, this specializes to condition~(ii) in Theorem \ref{thunr3intro}, where a permutation $G$-module is a $G$\nobreakdash-module that is a direct sum of $G$-modules isomorphic to $\Z[G]$ or to the trivial $G$-module $\Z$. In view of the Hochschild--Serre spectral sequence (\ref{HochschildSerre}), condition (ii) is equivalent, for smooth projective $\C$-rational varieties satisfying (i), to the assertion that the pull-back $\Br(\R)\to \Br(X)$ is an isomorphism, that is, to the triviality of the real analogue of the Artin--Mumford invariant.

Soul\'e and Voisin observed in \cite[Lemma 1]{soulevoisin} that the validity of the integral Hodge conjecture for $1$-cycles is a necessary condition for the $\C$-rationality of a smooth projective variety over $\C$. Condition (iii) is an analogue over $\R$ of this condition, in which the Borel\nobreakdash--Haefliger cycle class map $\cl_{\R}:\CH_1(X)\to H_1(X(\R),\Z/2\Z)$, defined in~\cite{BH}, associates with an integral curve $j:Z\hookrightarrow X$ with normalization $\pi:\widetilde{Z}\to Z$ the homology class $(j\circ \pi)_*[\widetilde{Z}(\R)]\in H_1(X(\R),\Z/2\Z)$.

That condition~(iii) holds for $\R$-rational varieties was already noticed, in the stronger form of an approximation theorem, by Bochnak and Kucharz \cite[Theorem~1.1]{BK}. 
It is possible that condition (iii)  is satisfied for all smooth projective rationally connected threefolds (see the more general \cite[Question~3.4]{BW}). Since this applies to $\C$-rational threefolds, this would allow one to remove (iii) from the statement of Theorem \ref{thunr3intro}.
Condition~(iii) is known to hold if $X$ is birational to a conic bundle over a $\C$-rational surface \cite[Corollary 6.5]{BW2}, or to a del Pezzo fibration of degree $\delta\in\{9,8,7,6,5,3\}$  over $\bP^1_{\R}$ \cite[Theorem 8.1 and Proposition 8.4]{BW2}.

\subsection*{Notation and conventions}

We fix a field $k$. Everywhere except in part of~\S\ref{secconic}, we assume that $k$ is perfect. We fix an algebraic closure $\overline{k}$
of~$k$ and let $\Gamma_k=\Aut(\overline{k}/k)$ be the absolute Galois group of $k$.
A \emph{variety} over $k$ is a separated scheme of finite type over $k$.
If $X$ is a variety over $k$, we let $\CH^c(X_{\overline{k}})_{\alg}\subset\CH^c(X_{\overline{k}})$ be the subgroup of algebraically trivial codimension $c$ cycle classes.

If $M$ is an abelian group and $n$ is an integer, we let $M[n]\subset M$ be the $n$-torsion subgroup.
If $\ell$ is a prime number we will consider the subgroup $M\{l\}:=\varinjlim_{\nu}M[\ell^\nu]$ of $\ell$-primary torsion of $M$ and the $\ell$-adic Tate module $T_{\ell}M:=\varprojlim_{\nu} M[\ell^\nu]$ of $M$.
If $M$ is a free $\Z_\ell$\nobreakdash-module (resp.\ $\Z$\nobreakdash-module) of finite rank,
we let $M^\vee = \Hom(M,\Z_\ell)$ (resp.\ $M^\vee=\Hom(M,\Z)$).

When $k=\R$, we set $G:=\Gamma_{\R}\simeq \Z/2\Z$, generated by the complex conjugation $\sigma\in G$. For $j\in\Z$, we consider the $G$-module $\Z(j):=(\sqrt{-1})^j\Z\subset \C$, and set $M(j):=M\otimes_{\Z}\Z(j)$ for any $G$-module $M$.

\subsection*{Acknowledgements}
We thank Olivier Piltant for explaining to us how to use embedded resolution of singularities for surfaces in the proof of Proposition \ref{resindet} and the referee for their careful work.

\section{Intermediate Jacobians}
\label{secij}

  In Section \ref{secij}, we study intermediate Jacobians of smooth projective threefolds.

\subsection{Principally polarized abelian varieties}
\label{ppavk}

A \textit{principally polarized abelian variety} (ppav) over $k$ is a pair $(A,\theta)$ consisting of an abelian variety $A$ over~$k$ and of a class $\theta\in\NS(A_{\overline{k}})^{\Gamma_k}$ induced by an ample line bundle on~$A_{\overline{k}}$ whose associated isogeny $A_{\overline{k}}\to \widehat{A}_{\overline{k}}$ (see \cite[Corollary 5 p.~131]{Mumford}) is an isomorphism. A morphism $p:(A',\theta')\to (A,\theta)$ of ppavs over $k$ is a (necessarily injective) morphism $p:A'\to A$ of abelian varieties such that $p^*\theta=\theta'$. One says that $(A',\theta')$ is a \textit{sub-ppav} of~$(A,\theta)$.

Let $C$ be a smooth projective curve over $k$, and let $(C_i)_{i\in I}$ be the connected components of $C_{\overline{k}}$. The \textit{Jacobian} $J^1C$ of $C$ is the identity connected component of the Picard scheme of $C$. It parametrizes line bundles on $C$ that have degree $0$ on all of the~$C_i$. 
There is a natural isomorphism $J^1C_{\overline{k}}\xrightarrow{\sim}\prod_{i\in I}J^1C_i$. Denoting by $\Theta_i\subset J^1C_i$ a theta divisor and by $p_i:J^1C_{\overline{k}}\to J^1C_i$ the projection, the line bundle $\bigotimes_{i\in I} p_i^*\sO_{J^1C_i}(\Theta_i)$
endows $J^1C$ with the structure of a ppav over~$k$.

A ppav over $k$ is \textit{indecomposable} if it is non-zero and is not isomorphic to the product
of two non-zero ppavs.
A ppav $(A,\theta)$ over $k$ is isomorphic to the product of its indecomposable sub-ppavs.
This is proved in \cite[Lemma 3.20, Corollary~3.23]{CG} if $k=\C$, and the proof still works if $k=\overline{k}$ as explained in \cite[Lemma 10]{Murrecubic}. One deduces the result in general by Galois descent: the indecomposable sub-ppavs of $(A,\theta)$, viewed over $\overline{k}$, are exactly the products of a $\Gamma_k$-orbit of indecomposable sub-ppavs of $(A_{\overline{k}},\theta)$.
We deduce that any morphism $p:(A',\theta')\to (A,\theta)$ of ppavs over $k$ induces a decomposition $(A,\theta)\simeq (A',\theta')\times (A'',\theta'')$ of $(A,\theta)$ as a product of ppavs over $k$.

The Jacobian $J^1C$ of a smooth projective connected (but not necessarily geometrically connected) curve $C$ over $k$ is indecomposable. If $k=\overline{k}$, this follows from the irreducibility of the theta divisor. In general, the connected components $C_i$ of $C_{\overline{k}}$ are permuted transitively by $\Gamma_k$ because $C$ is connected, so that the factors $J^1C_i$ of the decomposition $J^1C_{\overline{k}}=\prod_i J^1C_i$ as a product of indecomposable ppavs over $\overline{k}$ are permuted transitively by $\Gamma_k$, showing that $J^1C$ is indecomposable.

\subsection{Codimension \texorpdfstring{$2$}{2} algebraic cycles}
\label{cod2}
In this paragraph, we study substitutes over $k$ for the complex Abel--Jacobi map, with an emphasis on codimension $2$ cycles.

\subsubsection{Murre's intermediate Jacobian}
\label{AJMurre}

With a smooth projective variety $X$ over a perfect field~$k$ is associated an abelian variety $\Ab^2X$
 over~$k$, called the \textit{algebraic representative for algebraically trivial codimension $2$ cycles on~$X$}. (The construction of Murre over $\overline{k}$ \cite[Theorem~A~p.~226]{Murre}, as corrected by Kahn \cite{KahnMurre}, has been shown by Achter, Casalaina-Martin and Vial to descend to any perfect field \cite[Theorem~4.4]{ACMV1}.) It is characterized by the existence of a surjective $\Gamma_k$-equivariant map
\begin{equation}
\label{AJk}
\phi^2_X:\CH^2(X_{\overline{k}})_{\alg}\to \Ab^2X(\overline{k})
\end{equation}
that is initial among regular homomorphisms with values in an abelian variety over~$\overline{k}$ (see \cite[Definition 3.1, Theorem 4.4]{ACMV1}).

Let~$X$ and~$Y$ be smooth projective varieties over $k$ and $g:Y_{\overline{k}} \to X_{\overline{k}}$
be a morphism over~$\overline{k}$.
By the universal property of $\phi^2_X$, 
the composition $\phi^2_Y\circ g^*:\CH^2(X_{\overline{k}})_{\alg}\to \Ab^2Y(\overline{k})$ factors
as $g^+(\overline{k})\circ\phi^2_X$
for a unique morphism $g^+:\Ab^2X_{\overline{k}}\to \Ab^2Y_{\overline{k}}$.
Since $\phi^2_X$ is surjective and~$\phi^2_X$ and~$\phi^2_Y$ are $\Gamma_k$\nobreakdash-equivariant, the map
$g \mapsto g^+(\overline{k})$ is $\Gamma_k$\nobreakdash-equivariant,
and hence so is the map $g \mapsto g^+$.
In particular, if $g=f_{\overline{k}}$ for a morphism $f:Y\to X$ of varieties over~$k$,
then $g^+$ descends to a morphism $f^+:\Ab^2X\to \Ab^2Y$
of abelian varieties over~$k$.

The same argument shows that if $X$ and $Y$ are equidimensional of the same dimension and $f:Y\to X$ is a morphism
of varieties over~$k$, there exists a unique morphism $f_+:\Ab^2Y\to \Ab^2X$ of abelian varieties over $k$ such that $\phi^2_X\circ f_*=f_+(\overline{k})\circ\phi^2_Y$.

\subsubsection{Bloch's Abel--Jacobi map}
If $X$ is a smooth projective variety over $k$, Bloch has defined for all prime numbers~$\ell$ that are  invertible in $k$ and all $c\geq 0$ a morphism
\begin{equation}
\label{AJB}
\lambda^c:\CH^c(X_{\overline{k}})\{\ell\}\to H^{2c-1}_{\et}(X_{\overline{k}},\Q_{\ell}/\Z_{\ell}(c))
\end{equation}
called \textit{Bloch's} $\ell$-\textit{adic Abel--Jacobi map} \cite[\S2]{Bloch}, which is $\Gamma_k$-equivariant  by construction and compatible with the action of correspondences \cite[Proposition 3.5]{Bloch}. 
The map $\lambda^c$ is bijective if $c=1$ by Kummer theory \cite[Proposition 3.6]{Bloch}
 and injective if $c=2$ as a consequence of the Merkurjev--Suslin theorem \cite[Corollaire 4]{Torsion}.
The composition of $\lambda^c$ with the last arrow in the exact sequence
$$0\to H^{2c-1}_{\et}(X_{\overline{k}},\Z_{\ell}(c))\otimes\Q_\ell/\Z_\ell\to H^{2c-1}_{\et}(X_{\overline{k}},\Q_{\ell}/\Z_{\ell}(c))\to H^{2c}_{\et}(X_{\overline{k}},\Z_{\ell}(c))$$
is, up to a sign, the $\ell$-adic cycle class map \cite[Corollaire 4]{Torsion}.
Since the cycle class map vanishes on algebraically trivial cycles, $\lambda^c$ restricts to a $\Gamma_k$\nobreakdash-equivariant map 
\begin{equation}
\label{restrAJB}
\lambda^c:\CH^c(X_{\overline{k}})_{\alg}\{\ell\}\to H^{2c-1}_{\et}(X_{\overline{k}},\Z_{\ell}(c))\otimes\Q_\ell/\Z_\ell.
\end{equation}
This map is obviously still injective if $c\leq 2$, and it is still surjective if $c=1$: indeed, a codimension $1$ algebraic cycle of $\ell$-primary torsion that has trivial $\ell$-adic cycle class is algebraically trivial, in view
of the inclusion $\NS(X_{\overline{k}})\otimes \Z_{\ell} \subset H^2_{\et}(X_{\overline{k}},\Z_{\ell}(1))$ induced by the Kummer exact sequence
(see \cite[(2) p.~485]{SkZ}).

\subsubsection{Varieties with few zero-cycles}
We will use the following classical definition.

\begin{Def}
\label{defsuppi}
If $X$ is a smooth projective variety over $k$, we say that $\CH_{0}(X)_{\Q}$ is \textit{supported in dimension} $i$ if there exists a closed subvariety $V\subset X$ of dimension~$\leq i$ such that for all algebraically closed field extensions $k\subset\Omega$, the push-forward map $\CH_0(V_\Omega)_\Q\to \CH_0(X_\Omega)_\Q$ is surjective. 
\end{Def}

\begin{lem}
\label{lembir}
Let $f:X\dashrightarrow X'$ be a birational map of smooth projective varieties over $k$. If $\CH_{0}(X)_{\Q}$ is supported in dimension $i$, then so is $\CH_{0}(X')_{\Q}$.
\end{lem}

\begin{proof}
Let $V$ be as in Definition \ref{defsuppi}, let $\Gamma\subset X\times X'$ be the closure of the graph of $f$, let $W\subset \Gamma$ be a subvariety of dimension~$\leq i$ dominating $V$, and let $V'\subset X'$ be the image of~$W$. Let $k\subset\Omega$ be an algebraically closed field extension.
As any closed point of~$\Gamma_\Omega$ 
can be moved, by a rational equivalence, to any dense open subset of~$\Gamma_\Omega$
(choose a general curve passing through the point and normalize it),
the push-forward maps $\CH_0(\Gamma_\Omega)_\Q\to \CH_0(X_\Omega)_\Q$ and $\CH_0(\Gamma_\Omega)_\Q\to \CH_0(X'_\Omega)_\Q$ are compatible with the isomorphism  $\CH_0(X_\Omega)_\Q\simeq \CH_0(X'_\Omega)_\Q$ described in \cite[Example 16.1.11]{Fulton}. Hence the surjectivity of the push-forward $\CH_0(W_\Omega)_\Q\to \CH_0(X_\Omega)_\Q$ implies that of $\CH_0(W_\Omega)_\Q\to \CH_0(X'_\Omega)_\Q$, hence of $\CH_0(V'_\Omega)_\Q\to \CH_0(X'_\Omega)_\Q$.
\end{proof}

Codimension $2$ algebraic cycles on varieties with small Chow groups of zero-cycles behave particularly well. 
The following proposition applies for instance if $X_{\overline{k}}$ is rationally chain connected \cite[IV Definition 3.2]{Kollarbook}.

\begin{prop}
\label{decodiag}
Let $X$ be a smooth projective variety over $k$ such that $\CH_{0}(X)_{\Q}$ is supported in dimension $1$.
Then the following hold:
\begin{enumerate}[(i)]
\item The morphism $\phi^2_X:\CH^2(X_{\overline{k}})_{\alg}\to \Ab^2X(\overline{k})$ of (\ref{AJk}) is bijective.
\item The morphism $\lambda^2:\CH^2(X_{\overline{k}})_{\alg}\{\ell\}\to H^{3}_{\et}(X_{\overline{k}},\Z_{\ell}(2))\otimes\Q_\ell/\Z_\ell$ of (\ref{restrAJB}) is bijective for all prime numbers $\ell$ invertible in $k$.
\end{enumerate}
\end{prop}

\begin{proof}
This follows from the decomposition of the diagonal technique of Bloch and Srinivas \cite{BS}.
Assertion (i) is \cite[Theorem 1 (i) and its proof]{BS}, where one may replace the hypothesis on resolution of singularities by de Jong's alteration theorem \cite[Theorem 4.1]{dJ}.
Indeed, in the notation of \cite[proof of Theorem~1~(i)]{BS},
if~$\widetilde{D}$ is now allowed to be any smooth projective variety over~$\overline{k}$
endowed with a surjective generically finite map
of degree~$m \geq 1$ to~$D$, one may replace~$\Gamma_1$, $\Gamma_2$, $N$ with $m\Gamma_1$, $m\Gamma_2$, $mN$
to ensure that $[\Gamma_2]_*$ still factors through $\Pic^0(\widetilde{D})$.
Assertion~(ii) is proved in exactly the same way as~(i). More precisely, we have already seen that the map (\ref{restrAJB}) is injective if $c=2$ with no hypothesis on~$X$. Using the fact that (\ref{restrAJB}) is surjective if $c=1$, the decomposition of the diagonal argument shows that its cokernel is $N$-torsion for some $N>0$. Being a quotient of $H^{3}_{\et}(X_{\overline{k}},\Z_{\ell}(2))\otimes\Q_\ell/\Z_\ell$ it is also $N$-divisible, hence it vanishes.
\end{proof}

\subsection{The intermediate Jacobian}
\label{secijk}

Let $X$ be a smooth projective threefold over~$k$ such that $\CH_{0}(X)_{\Q}$ is supported in dimension $1$. 
Applying the $\ell$\nobreakdash-adic Tate module functor~$T_\ell$ to the morphisms (\ref{AJk}) and (\ref{restrAJB}), which are bijective by Proposition \ref{decodiag}, taking the identification
$T_\ell \Ab^2X(\overline{k})=H^1_{\et}(\Ab^2X_{\overline{k}},\Z_\ell)^\vee$ into account, and using,
for $M=H^{3}_{\et}(X_{\overline{k}},\Z_{\ell}(2))$, the isomorphism $M/(M\{\ell\})\xrightarrow{\sim}T_{\ell}(M\otimes\Q_\ell/\Z_\ell)$, valid for all finitely generated $\Z_{\ell}$-modules $M$, yields an isomorphism
\begin{equation}
\label{isoH1H3}
T_{\ell}(\lambda^2\circ(\phi_X^2)^{-1}):H^1_{\et}(\Ab^2X_{\overline{k}},\Z_\ell)^\vee\to  H^{3}_{\et}(X_{\overline{k}},\Z_{\ell}(2))/(\torsion).
\end{equation}

We will consider the following property of the smooth projective threefold $X$ over~$k$ (under the hypothesis that $\CH_{0}(X)_{\Q}$ is supported in dimension $1$).

\begin{property}
\label{assum}
There exists $\theta\in\NS(\Ab^2X_{\overline{k}})$ satisfying the following assertions.
\begin{enumerate}[(i)]
\item For all prime numbers $\ell$ invertible in $k$, the image $c_{1,\ell}(-\theta)$ of $-\theta$ by the $\ell$-adic first Chern class
$$c_{1,\ell}:\NS(\Ab^2X_{\overline{k}})\inj H^2_{\et}(\Ab^2X_{\overline{k}},\Z_{\ell}(1))=\Big(\bigwedge^2H^1_{\et}(\Ab^2X_{\overline{k}},\Z_\ell)\Big)(1)$$
corresponds, via the isomorphism (\ref{isoH1H3}),
 to the cup product map 
$$\bigwedge^2 H^3_{\et}(X_{\overline{k}},\Z_{\ell}(2))\xrightarrow{\smile} H^6_{\et}(X_{\overline{k}},\Z_{\ell}(4))\xrightarrow{\deg}\Z_{\ell}(1).$$
\item The class $\theta\in\NS(\Ab^2X_{\overline{k}})$ is a principal polarization of $\Ab^2X_{\overline{k}}$.
\end{enumerate}
\end{property}

Property \ref{assum} only depends on $X_{\overline{k}}$. Hence, whenever we need to verify it, we may replace $k$ with~$\overline{k}$ and $X$ with $X_{\overline{k}}$.

A class $\theta$ as in Property~\ref{assum}~(i) is unique since $c_{1,\ell}$ is injective (by \cite[(2) p.~485]{SkZ} and since $\NS(\Ab^2X_{\overline{k}})$ has no torsion).
 Being unique, it must be $\Gamma_k$-invariant, by the $\Gamma_k$-equivariance of $c_{1,\ell}$, of $\phi^2_X$, of $\lambda^2$ and of the cup product map. Consequently, if $X$ satisfies Property \ref{assum}, then $(\Ab^2X,\theta)$ is a ppav over~$k$, which we denote by~$J^3X$ and call the \textit{intermediate Jacobian} of $X$.

Although we will not use it in the sequel, the following proposition, which applies if $k$ has characteristic $0$ and $X_{\overline{k}}$ is rationally connected, is a motivation for Property~\ref{assum} (and its proof is a justification for the notation $J^3X$).

\begin{prop}
\label{propRC}
A smooth projective threefold over a field $k$ of characteristic $0$ such that $\CH_{0}(X)_{\Q}$ is supported in dimension $0$ satisfies Property~\ref{assum}.
\end{prop}

\begin{proof}
Since Property \ref{assum} only depends on $X_{\overline{k}}$, we may assume that $k=\overline{k}$. By the Lefschetz principle, using in particular that the formation of $\Ab^2X$ commutes with extensions of algebraically closed fields  of characteristic $0$ \cite[Theorem~3.7]{ACMV1}, one may further assume that $k=\C$. 
By decomposition of the diagonal, one has $H^0(X,\Omega^1_X)=H^0(X,\Omega_X^3)=0$ (see \cite[Corollary~1.10]{Laterveer}).

Let us temporarily denote by $J^3X$ Griffiths' intermediate Jacobian of $X$, that is, the complex torus $J^3X:=H^2(X,\Omega_X^1)/\Ima(H^3(X(\C),\Z(2)))$ (see~\cite{CG}). The transcendental Abel--Jacobi map $\AJ^2:\CH^2(X)_\alg\to J^3X(\C)$ is surjective by its compatibility with Bloch's $\ell$-adic Abel--Jacobi map (\ref{AJB}) \cite[Proposition~3.7]{Bloch}, and by Proposition~\ref{decodiag}~(ii). It then follows from \cite[Theorem~C~p.~229]{Murre} that $\AJ^2$ satisfies the universal property (\ref{AJk}) of $\phi^2_X$, yielding an identification $J^3X\simeq \Ab^2X$.

Let  $\gamma\in H^2(J^3X(\C),\Z(1))=(\bigwedge^2 H_1(J^3X(\C),\Z)^\vee)(1)$ be such that $-\gamma$ corresponds, under the identification $H_1(J^3X(\C),\Z)=H^3(X(\C),\Z(2))/(\torsion)$,
to the cup product map $\bigwedge^2 H^{3}(X(\C),\Z(2))\xrightarrow{\smile} H^6(X(\C),\Z(4))\xrightarrow{\deg}\Z(1)$. 
We claim that $\gamma$ is the first Chern class of a principal polarization $\theta$ on $\Ab^2(X)=J^3X$. To see it, one has to show that $\gamma$ is unimodular and that its associated Hermitian form is positive definite (see \cite[\S 2.1, \S 4.1]{BL}). These assertions are respectively consequences of Poincar\'e duality and of the Hodge--Riemann relations \cite[Th\'eor\`eme~6.32]{Voisin}.

That $\theta$ has the required properties follows from comparison between $\ell$-adic and Betti cohomology, from the fact that we identified $\Ab^2X$ and $J^3X$ using $\phi^2_X$ and~$\AJ^2$, and from the compatibility of $\AJ^2$ and $\lambda^2$ \cite[Proposition~3.7]{Bloch}.
\end{proof}

We do not know if Proposition \ref{propRC} always holds if $k$ has positive characteristic.
We will verify it if $X$ is $\overline{k}$-rational in Corollary \ref{CGk}.

\begin{rem}
Over the field $k=\R$ of real numbers, there is a more general and much easier way to construct intermediate Jacobians than Proposition \ref{propRC}. Indeed, let~$X$ be a smooth projective threefold over $\R$ such that $H^0(X,\Omega^1_X)=H^0(X,\Omega_X^3)=0$ and let $J^3X_\C$ denote
the intermediate Jacobian of $X_\C$ constructed by transcendental means
as in~\cite{CG}.
We recall that the complex analytic space $J^3X_\C(\C)$ is by definition the cokernel
of the composition
\begin{equation}
\label{compojacint}
H^3(X(\C),\Z(2)) \to H^3(X(\C),\C) \to H^2(X_\C,\Omega^1_{X_\C})
\end{equation}
of the map induced by the inclusion $\Z(2)\subset \C$ with the projection stemming from the Hodge decomposition.
On $H^3(X(\C),\C)$, one can consider the $\C$-linear involution~$F_\infty$ induced by the complex conjugation
involution of~$X(\C)$
and the two $\C$\nobreakdash-antilinear involutions~$F_{\mathrm{B}}$ and~$F_{\mathrm{dR}}$
corresponding, respectively, to the real structures
$H^3(X(\C),\C)=H^3(X(\C),\R) \otimes_{\R} \C$
and $H^3(X(\C),\C)=H^3_{\mathrm{dR}}(X/\R)\otimes_{\R} \C$.
They all commute, and
are related by the formula $F_{\mathrm{dR}} \circ F_{\mathrm{B}} \circ F_{\infty}=\Id$
\cite[Proposition~1.4]{DeligneCorvallis}. It follows that $F_{\mathrm{dR}}$ stabilises the image of the first arrow of (\ref{compojacint}). Also denoting by $F_{\mathrm{dR}}$ the $\C$-antilinear involution of $H^2(X_\C,\Omega^1_{X_\C})$ associated with the real structure $H^2(X_\C,\Omega^1_{X_\C})=H^2(X,\Omega^1_{X})\otimes_{\R}\C$ and noting that the second arrow of (\ref{compojacint}) is $F_{\mathrm{dR}}$-equivariant, we deduce that $F_{\mathrm{dR}}$ stabilises the image of (\ref{compojacint}) and thus equips~$J^3 X_{\C}(\C)$ with an antiholomorphic
involution. The polarization of~$J^3X_{\C}$, being given by the opposite of the cup product
map $H^3(X(\C),\Z(2)) \times H^3(X(\C),\Z(2)) \to \Z(1)$, is preserved by this involution
since the cup product is equivariant with respect to~$F_{\mathrm{B}}$ and to~$F_{\infty}$.
Hence~$J^3 X_{\C}$ descends to a ppav $J^3X$ over~$\R$, which is the sought for intermediate Jacobian of $X$. This method avoids the use of the deep results of Bloch \cite{Bloch} and Murre \cite{Murre}, and would be sufficient for the proof of Theorem \ref{main}.
\end{rem}

\subsection{Birational behaviour}
\label{geomrat}

We now show that the validity of Property \ref{assum} is a birational invariant of smooth projective threefolds over $k$. Recall that the assertion that $\CH_{0}(X)_{\Q}$ is supported in dimension $1$, which is required for Property \ref{assum} to make sense, is a birational invariant of smooth projective varieties over $k$ by Lemma~\ref{lembir}.

\begin{thm}
\label{thbir}
Let $X$ and $Y$ be birational smooth projective threefolds over $k$ such that $\CH_{0}(Y)_{\Q}$ is supported in dimension $1$. If $Y$ satisfies Property~\ref{assum}, then so does $X$. Moreover, there exist smooth projective curves $C$ and $C'$ over $k$ and an isomorphism $J^3Y\times J^1C\simeq J^3X\times J^1C'$ of ppavs over $k$.
\end{thm}

It follows that one can associate with any smooth projective $\overline{k}$\nobreakdash-rational threefold $X$ over $k$ a ppav $J^3X$ over $k$ that gives rise to an obstruction to the $k$-rationality of $X$ extending \cite[Corollary 3.26]{CG}. 

\begin{cor}
\label{CGk}
A smooth projective $\overline{k}$-rational threefold $X$ over $k$ satisfies Property~\ref{assum}. 
If $X$ is moreover $k$-rational, then its intermediate Jacobian $J^3X$ is isomorphic, as a ppav over $k$, to the Jacobian of a smooth projective curve over~$k$.
\end{cor}

\begin{proof}
To verify the first assertion, we may work over $\overline{k}$ as Property \ref{assum} only depends on $X_{\overline{k}}$. It then follows from Theorem \ref{thbir} applied with $Y=\bP^3_{\overline{k}}$. Indeed, $\bP^3_{\overline{k}}$ satisfies Property \ref{assum}: since $ \CH^2(\bP^3_{\overline{k}})_{\alg}=0$, one even has $\Ab^2\bP^3_{\overline{k}}=0$.

To show the second assertion, we apply Theorem \ref{thbir} with $Y=\bP^3_k$. By the above, one has ${J^3\bP^3_k=0}$, and we obtain an isomorphism $J^1C\simeq J^3X\times J^1C'$ of ppavs over~$k$ for some smooth projective curves $C$ and $C'$ over $k$. 
Since the indecomposable factors of $J^1C$ are Jacobians of smooth projective connected curves over $k$, the uniqueness of the decomposition of $J^3X$ as a product of indecomposable factors (see \S\ref{ppavk}) shows that~$J^3X$ is itself a product of Jacobians of smooth projective connected curves over~$k$, hence is the Jacobian of a smooth projective curve over~$k$.
\end{proof}

We first study the behaviour of Property \ref{assum} under birational morphisms.

\begin{lem}
\label{lemdom}
Let $f:Y\to X$ be a birational morphism of smooth projective threefolds over $k$ such that $\CH_{0}(Y)_{\Q}$ is supported in dimension $1$. If $Y$ satisfies Property~\ref{assum}, then so does $X$
and moreover there is
an isomorphism $J^3Y \simeq J^3X \times B$ of ppavs over~$k$
for some ppav~$B$ over~$k$.
\end{lem}

\begin{proof}
Let $\theta_Y\in \NS(\Ab^2Y_{\overline{k}})$ be the class given by Property~\ref{assum} for $Y$. Define $\theta_X:=(f^+)^*\theta_Y\in \NS(\Ab^2X_{\overline{k}})$. We first remark that $\theta_X$ satisfies the condition of Property~\ref{assum}~(i) for $X$.  Indeed, fixing a prime number $\ell$ invertible in $k$, this follows from the commutativity of the diagram (see \S\ref{secijk})
\begin{align*}
\xymatrix
@R=0.5cm
{
H^1_{\et}(\Ab^2X_{\overline{k}},\Z_\ell)^\vee=T_\ell\Ab^2X(\overline{k})\ar@<-3em>[d]^{((f^+)^*)^\vee} & T_\ell\CH^2(X_{\overline{k}})_{\alg}\ar[l]^(.34){\sim}_(.34){T_\ell\phi^2_X}\ar[r]_(.44){\sim}^(.44){T_\ell\lambda^2}\ar[d]^{f^*} &H^3_{\et}(X_{\overline{k}},\Z_\ell(2))/(\tors)\ar[d]^{f^*}  \\
H^1_{\et}(\Ab^2Y_{\overline{k}},\Z_\ell)^\vee=T_\ell\Ab^2Y(\overline{k})& T_\ell\CH^2(Y_{\overline{k}})_{\alg}\ar[l]^(.34){\sim}_(.34){T_\ell\phi^2_Y}\ar[r]_(.44){\sim}^(.44){T_\ell\lambda^2}&H^3_{\et}(Y_{\overline{k}},\Z_\ell(2))/(\tors)
}
\end{align*}
since $f^*:H^3_{\et}(X_{\overline{k}},\Z_\ell(2))\to H^3_{\et}(Y_{\overline{k}},\Z_\ell(2))$ respects the cup product pairing.

Let us turn to Property~\ref{assum}~(ii) for~$X$.
One has $f_+\circ f^+=\Id: \Ab^2X\to \Ab^2X$ since $f_*\circ f^*=\Id:\CH^2(X_{\overline{k}})\to \CH^2(X_{\overline{k}})$ and since $\phi^2_X$ is surjective. A natural isomorphism of abelian varieties
\begin{align}
\label{eq:decompositionab2yx}
\Ab^2Y\simeq f^+(\Ab^2X)\times \Ker(f_+)
\end{align}
results.  It follows from the above diagram and from the same diagram
with $f_*$, $f_+$ replaced by $f^*$, $f^+$ (and vertical arrows reversed)
that applying the rational $\ell$\nobreakdash-adic Tate module functor
to~\eqref{eq:decompositionab2yx} yields, via~\eqref{isoH1H3}, the decomposition
\begin{align}
\label{eq:decompositionh3}
H^3(Y_{\overline{k}},\Q_\ell(2)) \simeq f^*H^3(X_{\overline{k}},\Q_\ell(2)) \times \Ker(f_*)
\end{align}
stemming from the equality $f_* \circ f^*=\Id:H^3(X_{\overline{k}},\Q_\ell(2))\to H^3(X_{\overline{k}},\Q_\ell(2))$.
By the projection formula, the latter decomposition is orthogonal with respect
to the cup product pairing on $H^3(Y_{\overline{k}},\Q_\ell(2))$. Hence the decomposition
of $T_\ell(\Ab^2 Y)$
induced by~\eqref{eq:decompositionab2yx}
is orthogonal with respect
to the
pairing
$c_{1,\ell}(\theta_Y)$.  Equivalently, if $p:\Ab^2 Y \to \Ker(f_+)$ denotes the
projection, then $\theta_Y=(f_+)^*\theta_X + p^*\theta$ for some $\theta \in \NS(\Ker(f_+))$.
As $\theta_Y$ is a principal polarization, so must be $\theta_X$ and $\theta$.
\end{proof}

\begin{lem}
\label{lembl}
Let $X$ be a smooth projective threefold over $k$ with $\CH_{0}(X)_{\Q}$ supported in dimension $1$. Let $f:Y\to X$ be the blow-up of a smooth subscheme~${Z\subset X}$ and let $C$ be the union of the one-dimensional components of $Z$. If~$X$ satisfies Property \ref{assum}, then so does~$Y$
and moreover there is
an isomorphism $J^3Y \simeq J^3X \times J^1C$ of ppavs over~$k$.
\end{lem}

\begin{proof}
We first construct an isomorphism of abelian varieties $\Ab^2Y\simeq\Ab^2X\times J^1C$.

Let $i:E\inj Y$ be the inverse image of $C$ in $Y$.
Consider the correspondence $z:=(i,f|_E)_*E\in \CH^2(Y\times C)$.
The existence of the Poincar\'e divisor on $C_{\overline{k}}\times J^1C_{\overline{k}}$ inducing the natural bijection $\phi^1_C:\CH^1(C_{\overline{k}})_{\alg}\to J^1C(\overline{k})$ and the fact that $\phi^2_{Y}$ is regular in the sense of \cite[Definition 1.6.1]{Murre}
show the existence of a morphism $z^+:J^1C_{\overline{k}}\to \Ab^2Y_{\overline{k}}$ such that $z^+(\overline{k})\circ \phi^1_C=\phi^2_{Y}\circ z^*:\CH^1(C_{\overline{k}})_{\alg}\to \Ab^2Y(\overline{k})$. Since $\phi^1_C$, $\phi^2_{Y}$ and $z^*$ are $\Gamma_k$-equivariant, so is $z^+(\overline{k})$, showing that $z^+$ descends to a morphism $z^+:J^1C\to \Ab^2Y$ defined over $k$. A similar argument shows the existence of a morphism $z_+:\Ab^2Y\to J^1C$ of abelian varieties over $k$ such that $\phi^1_C\circ z_*= z_+(\overline{k})\circ\phi^2_{Y}:\CH^2(Y_{\overline{k}})_{\alg}\to J^1C(\overline{k})$.

Computing the Chow groups of a blow-up 
 \cite[Proposition~6.7~(e)]{Fulton} yields a canonically split short exact sequence
\begin{equation}
\label{blowupformula}
0\to \CH_1(Z_{\overline{k}})\to \CH_1(X_{\overline{k}})\times \CH_1(f^{-1}(Z)_{\overline{k}})\to \CH_1(Y_{\overline{k}})\to 0.
\end{equation}
As $f^{-1}(Z)$ is a projective bundle of relative dimension $\geq 1$ over $Z$, there is a canonical isomorphism
$\CH_1(f^{-1}(Z)_{\overline{k}})\simeq \CH_0(Z_{\overline{k}})\times \CH_1(Z_{\overline{k}})$ \cite[Theorem~3.3]{Fulton}. Combining it with (\ref{blowupformula}), we get an isomorphism $\CH_1(X_{\overline{k}})\times\CH_0(Z_{\overline{k}})\to \CH_1(Y_{\overline{k}})$. Identifying the arrows and restricting to algebraically trivial cycles
shows that $(f^*,z^*):\CH^2(X_{\overline{k}})_{\alg}\times \CH^1(C_{\overline{k}})_{\alg}\to \CH^2(Y_{\overline{k}})_{\alg}$ is an isomorphism with inverse
$(f_*,-z_*): \CH^2(Y_{\overline{k}})_{\alg}\to\CH^2(X_{\overline{k}})_{\alg}\times \CH^1(C_{\overline{k}})_{\alg}$. We deduce at once that $(f^+,z^+): \Ab^2X\times J^1C\to \Ab^2Y$ and $(f_+,-z_+):\Ab^2Y\to\Ab^2X\times J^1C$ are inverse isomorphisms of abelian varieties over $k$.

Let $\theta_C\in\NS(J^1C_{\overline{k}})$ be the canonical principal polarization. For all primes $\ell$ invertible in $k$, the class $c_{1,\ell}(\theta_C)\in H^2_{\et}(J^1C_{\overline{k}},\Z_\ell(1))=\Big(\bigwedge^2 H^1_{\et}(J^1C_{\overline{k}},\Z_{\ell})\Big)(1)$
corresponds, via the isomorphism $T_{\ell}(\lambda^1\circ(\phi_C^1)^{-1}):H^1_{\et}(J^1C_{\overline{k}},\Z_\ell)^\vee\to  H^{1}_{\et}(C_{\overline{k}},\Z_{\ell}(1))$, to the cup product map 
$\bigwedge^2 H^1_{\et}(C_{\overline{k}},\Z_{\ell}(1))\xrightarrow{\smile} H^2_{\et}(C_{\overline{k}},\Z_{\ell}(2))\xrightarrow{\deg}\Z_{\ell}(1)$.
To verify this classical fact, already used by Murre in \cite[\S 3.6]{Murrecubic}, one may to reduce to $k$ of characteristic~$0$ by lifting $C$ to such a field, then to $k=\C$ by the Lefschetz principle, where it follows from a transcendental computation (for which see \cite[\S\S 11.1--11.2]{BL}) after comparing $\ell$-adic and Betti cohomology.

Let $\theta_X\in \NS(\Ab^2X_{\overline{k}})$ be the class given by Property~\ref{assum} for $X$, and define $\theta_Y:=(f_+,-z_+)^*(\theta_X,\theta_C)$. We only need to show that $\theta_Y$ satisfies Property \ref{assum}~(i).
This follows from the above property of $\theta_C$ and from the commutativity of the two diagrams
\begin{align*}
\xymatrix
@R=0.5cm
@C=2cm
{
H^1_{\et}(\Ab^2X_{\overline{k}},\Z_{\ell})^\vee\times H^1_{\et}(J^1C_{\overline{k}},\Z_{\ell})^{\vee}\ar[d]_{}^{((f^+)^*,(z^+)^*)^\vee} & T_\ell\CH^2(X_{\overline{k}})_{\alg}\times T_\ell\CH^1(C_{\overline{k}})_{\alg}\ar[l]^(.46){\sim}_(.46){(T_\ell\phi^2_X,T_\ell\phi^1_C)}\ar[d]_{\wr}^{(f^{*}\hspace{-.2em},\hspace{.1em}z^*)}  \\
H^1_{\et}(\Ab^2Y_{\overline{k}},\Z_\ell)^\vee& T_\ell\CH^2(Y_{\overline{k}})_{\alg}\ar[l]^(.46){\sim}_(.46){T_\ell\phi^2_Y}
}
\end{align*}
\begin{align*}
\xymatrix
@R=0.5cm
@C=1.3cm
{T_\ell\CH^2(X_{\overline{k}})_{\alg}\times T_\ell\CH^1(C_{\overline{k}})_{\alg}\ar[r]_(.46){\sim}^(.46){(T_\ell\lambda^2,T_\ell\lambda^1)}\ar[d]_{\wr}^{(f^{*}\hspace{-.2em},\hspace{.1em}z^*)} &H^3_{\et}(X_{\overline{k}},\Z_\ell(2))/(\tors)\times H^1_{\et}(C_{\overline{k}},\Z_\ell(1)) \ar[d]_{}^{(f^{*}\hspace{-.2em},\hspace{.1em}z^*)}  \\ T_\ell\CH^2(Y_{\overline{k}})_{\alg}\ar[r]_(.46){\sim}^(.46){T_\ell\lambda^2}&H^3_{\et}(Y_{\overline{k}},\Z_\ell(2))/(\tors),
}
\end{align*}
since $f^*:H^3_{\et}(X_{\overline{k}},\Z_\ell(2))\to H^3_{\et}(Y_{\overline{k}},\Z_\ell(2))$ respects the cup product pairing and since
$z^*: H^1_{\et}(C_{\overline{k}},\Z_\ell(1))\to H^3_{\et}(Y_{\overline{k}},\Z_\ell(2))$ reverses it in the sense that \begin{equation}
\label{signcp}
\deg(z^*x\smile z^*y)=-\deg(x\smile y)\in\Z_\ell(1)
\end{equation}
for all $x,y\in H^1_{\et}(C_{\overline{k}},\Z_\ell(1))$. Identity (\ref{signcp}) was proved by Clemens and Griffiths in \cite[(3.12)]{CG} when $k=\C$. To check it, set $x':=(f|_E)^*x$ and $y':=(f|_E)^*y$,
so that $x',y' \in H^1_{\et}(E_{\overline{k}},\Z_\ell(1))$, and compute
$$ i_*x'\smile i_*y'=i_*(i^*i_*x'\smile y')=i_*(x'\smile y'\smile i^*i_*1)=i_*(x'\smile y'\smile c_1(\sO_E(-1))).$$
Since $z^*x\smile z^*y=i_*x'\smile i_*y'$, the projection formula yields
\phantom\qedhere
$$
\deg(z^*x\smile z^*y)=\deg(x'\smile y'\smile c_1(\sO_E(-1)))=-\deg(x \smile y).\eqno\qed$$
\end{proof}

To go further, we need a resolution of indeterminacies result going back to Abhyankar \cite{Abhyankar}, which we will use exactly as Murre in \cite[\S 3]{Murrecubic} (see also \cite{LP} for an application in a similar vein).

\begin{prop}
\label{resindet}
Let $f:Y\dashrightarrow X$ be a rational map of varieties over $k$ with $Y$ smooth quasi-projective of dimension $3$ and $X$ projective. Then there exists a composition $g:Z\to Y$ of blow-ups with smooth centers and a morphism $h:Z\to X$ such that $h=f\circ g$.
\end{prop}

\begin{proof} 
Let $\Gamma\subset Y\times_k X$ be the closure of the graph of $f$. Since the projection $\Gamma\to Y$ is projective and birational, it is the blow\nobreakdash-up of some coherent sheaf of ideals $\sI\subset\sO_{Y}$ \cite[II Theorem 7.17]{Hartshorne}. There exists a morphism $g:Z\to Y$ that is a composition of blow-ups with smooth centers such that the sheaf of ideals $\sI\sO_Z\subset\sO_Z$ is invertible \cite[(9.1.4)]{Abhyankar} (see \cite[Proposition~4.2]{CPI} or \cite[Theorem 5.9]{CJS} for modern references). By the universal property of a blow-up, $Z$ dominates $\Gamma$, and we let $h:Z\to X$ be the induced morphism.
\end{proof}

We may finally conclude the proof of Theorem \ref{thbir}.

\begin{proof}[Proof of Theorem \ref{thbir}]
Let $f:Y\dashrightarrow X$ be a birational map. By Proposition~\ref{resindet}, there exists a composition $g:Z\to Y$ of blow-ups with smooth centers and a morphism $h:Z\to X$ such that $h=f\circ g$.
By Lemmas \ref{lemdom} and \ref{lembl}, the varieties $Y$, $Z$, $X$ all satisfy Property~\ref{assum},
and we obtain an isomorphism
\begin{align}
\label{eq:iso1}
J^3X \times B \simeq J^3Y \times J^1C
\end{align}
of ppavs over~$k$ for some ppav~$B$ over~$k$ and some smooth projective curve~$C$ over~$k$.
The same reasoning applied to~$f^{-1}$ produces a ppav $B'$ over~$k$, a smooth projective curve~$C'$ over~$k$ and an isomorphism
\begin{align}
\label{eq:iso2}
J^3X \times J^1C' \simeq J^3Y \times B'
\end{align}
of ppavs over~$k$.
By the uniqueness of the decomposition of a ppav as a product of indecomposable factors,
and as the indecomposable factors of~$J^1C$ and~$J^1C'$ are Jacobians of smooth projective connected
curves over $k$ (see \S\ref{ppavk}), we deduce
from the isomorphism
$J^3Y \times B \times B' \simeq J^3Y\times J^1C \times J^1C'$
obtained by combining~\eqref{eq:iso1} and~\eqref{eq:iso2}
that~$B$ and~$B'$ are themselves Jacobians of smooth
projective curves over~$k$.  Thus~\eqref{eq:iso1} is the desired isomorphism.
\end{proof}

\section{Counterexamples to the L\"uroth problem}

We now explain how to use the intermediate Jacobians studied in Section \ref{secij} to construct examples of varieties over $k$ that are $\overline{k}$-rational but not $k$-rational.

\subsection{Twists}
\label{twists}

Our examples will be constructed as twists of $k$-rational varieties.
Let $X$ be a quasi-projective variety over~$k$. The \textit{twist} ${}_cX$ of $X$ by the $1$-cocycle $c=(c_{\gamma})_{\gamma\in\Gamma_k}\in Z^1(k,\Aut(X_{\overline{k}}))$ (see \cite[I 5.1 and III 1.3]{CohoGal}) is a variety over $k$ with an isomorphism $i: X_{\overline{k}}\simeq({}_cX)_{\overline{k}}$ such that $\gamma(i(x))=i(c_{\gamma}\cdot\gamma(x))$ for all $x\in X(\overline{k})$ and $\gamma\in\Gamma_k$. The twists of $X$ are exactly the varieties over $k$ that are $\overline{k}$-isomorphic to $X_{\overline{k}}$, and two twists ${}_cX$ and ${}_{c'}X$ of $X$ are isomorphic as varieties over $k$ if and only if $c=c'$ are cohomologous
\cite[III 1.3, Proposition 5]{CohoGal}. We denote by $[{}_cX]\in H^1(k,  \Aut(X_{\overline{k}}))$ the cohomology  class of $c$.

Similarly, if $(A,\theta)$ is a ppav over $k$ and $d\in Z^1(k,\Aut_{\ppav}(A_{\overline{k}},\theta))$, the twist $_{d}(A,\theta)$ of $(A,\theta)$ by $d$ is a ppav over $k$ such that $(A_{\overline{k}},\theta)\simeq ({}_{d}(A,\theta))_{\overline{k}}$, two cocycles give rise to isomorphic ppavs over $k$ if and only if they are cohomologous, and we set $[_{d}(A,\theta)]$ to be the image of $d$ in $H^1(k,\Aut_{\ppav}(A_{\overline{k}},\theta))$.

\begin{prop}
\label{twisttwist}
Let $X$ be a smooth projective threefold over~$k$. Assume that $\CH_{0}(X)_{\Q}$ is supported in dimension $1$ and that $X$ satisfies Property \ref{assum}. Let $\chi : \Aut(X_{\overline{k}})\to\Aut_{\ppav}(J^3X_{\overline{k}})$ be the $\Gamma_k$\nobreakdash-equivariant map $g\mapsto g^+$ (see \S\ref{AJMurre}). Then for all $c\in Z^1(k,\Aut(X_{\overline{k}}))$, one has $J^3({}_cX)\simeq{}_{\chi(c)}(J^3X)$ as ppavs over~$k$.
\end{prop}

\begin{proof}
The map $\phi^2_X$ considered in (\ref{AJk}) being $\Gamma_k$\nobreakdash-equivariant by hypothesis, and $\Aut(X_{\overline{k}})$\nobreakdash-equivariant by the functoriality of $J^3X$ (see \S\ref{AJMurre}), the composition
$$\CH^2(({}_cX)_{\overline{k}})_{\alg}\simeq\CH^2(X_{\overline{k}})_{\alg}\xrightarrow{\phi^2_X} J^3X(\overline{k})\simeq {}_{\chi(c)}(J^3X)(\overline{k}),$$
where the first and third arrows are respectively induced by the natural isomorphisms $X_{\overline{k}}\simeq({}_cX)_{\overline{k}}$ and $J^3X_{\overline{k}}\simeq({}_{\chi(c)}(J^3X))_{\overline{k}}$, is $\Gamma_k$-equivariant as the defects of $\Gamma_k$-equivariance of the first and third arrow compensate each other exactly. This yields an isomorphism $J^3({}_cX)\simeq{}_{\chi(c)}J^3X$ of ppavs over $k$ by the definition (\ref{AJk}) of $J^3({}_cX)$.
\end{proof}

\subsection{Quadratic twists}
\label{qtwists}

 If $(A,\theta)$ is a ppav over $k$, sending~$1$ to the automorphism $-\Id$ of~$(A_{\overline{k}},\theta)$ induces a $\Gamma_k$-equivariant morphism 
$\varphi:\Z/2\Z\to\Aut_{\ppav}(A_{\overline{k}},\theta)$, hence a map $\varphi: H^1(k,\Z/2\Z)= Z^1(k,\Z/2\Z)\to Z^1(k,\Aut_{\ppav}(A_{\overline{k}},\theta))$. 
For all $a\in H^1(k,\Z/2\Z)$, 
the ppav ${}_{\varphi(a)}(A,\theta)$ over $k$ is the \textit{quadratic twist} of $(A,\theta)$ by~$a$.

Our main result regarding the non-triviality of the Clemens--Griffiths invariant over $k$ is the following consequence of the Torelli theorem.

\begin{prop}
\label{Torelli}
Let $C$ be a smooth projective geometrically connected curve over~$k$, and let $a\in H^1(k,\Z/2\Z)$. Then the following conditions are equivalent:
\begin{enumerate}[(i)]
\item There exists a smooth projective curve $C'$ over $k$ such that ${}_{\varphi(a)}(J^1C)$ and $J^1C'$ are isomorphic as ppavs over $k$.
\item The class $a$ is trivial, or $C_{\overline{k}}$ has genus $0$ or $1$, or $C_{\overline{k}}$ is hyperelliptic.
\end{enumerate}
\end{prop}

\begin{proof}
We first prove that (ii)$\Rightarrow$(i).
If $a$ is trivial, then ${}_{\varphi(a)}(J^1C)\simeq J^1C$. If $C$ has genus~$0$, then ${}_{\varphi(a)}(J^1C)=0$ and if $C$ has genus $1$, then ${}_{\varphi(a)}(J^1C)\simeq J^1({}_{\varphi(a)}(J^1C))$.
  Suppose now that $C_{\overline{k}}$ is hyperelliptic of genus $\geq 2$. By \cite[Appendice, Th\'eor\`eme~4]{Serre}, one has a $\Gamma_k$-equivariant group isomorphism
$\Aut(C_{\overline{k}})\xrightarrow{\sim}\Aut_{\ppav}(J^1C_{\overline{k}})$ inducing a bijection
$H^1(k,\Aut(C_{\overline{k}}))\xrightarrow{\sim}H^1(k,\Aut_{\ppav}(J^1C_{\overline{k}}))$. The inverse image of the class $[{}_{\varphi(a)}(J^1C)]$ by this bijection corresponds to a twist $C'$ of $C$ (called a hyperelliptic twist) with the property that $J^1C'\simeq {}_{\varphi(a)}(J^1C)$ as ppavs over $k$. 

Assume now that (i) holds but that $C_{\overline{k}}$ has genus $\geq 2$ and is not hyperelliptic. Since $J^1C_{\overline{k}}$ and $J^1C'_{\overline{k}}$ are isomorphic ppavs over $\overline{k}$, the Torelli theorem shows that $C'_{\overline{k}}\simeq C_{\overline{k}}$. By \cite[Appendice, Th\'eor\`eme~4]{Serre}, one has a $\Gamma_k$-equivariant group isomorphism
$\Aut(C_{\overline{k}})\times \Z/2\Z\xrightarrow{\sim}\Aut_{\ppav}(J^1C_{\overline{k}})$ associating with an automorphism of~$C_{\overline{k}}$ the induced automorphism of $J^1C_{\overline{k}}$ and with $1\in\Z/2\Z$ the automorphism $-\Id$.
This yields a bijection
$H^1(k,\Aut(C_{\overline{k}}))\times H^1(k,\Z/2\Z)\xrightarrow{\sim}H^1(k,\Aut_{\ppav}(J^1C_{\overline{k}}))$. The images $[{}_{\varphi(a)}(J^1C)]$ and $[J^1C']$ of $([C],a)$ and $([C'],0)$ by this bijection coincide because ${}_{\varphi(a)}(J^1C)\simeq J^1C'$ as ppavs over $k$, showing that $a$ is trivial.
\end{proof}

\subsection{Conic bundles}
\label{secconic}

In this paragraph, we do not assume that $k$ is perfect.

Concrete varieties to which we may apply the above results are conic bundles.
Let $S$ be a smooth projective  $\overline{k}$-rational surface over $k$,
let $\mathcal{L}$ be an invertible sheaf on~$S$
and let $F\in H^0(S,\sL^{\otimes 2})$ be a non-zero section with smooth zero locus ${C:=\{F=0\}\subset S}$. 
We define $p:\mathbb{P}:=\mathbb{P}_S(\mathcal{L}^{-1}\oplus\mathcal{L}^{-1}\oplus\mathcal{O}_S)\to S$ as a rank~$2$ projective bundle over $S$
in the sense of Grothendieck, with tautological bundle~$\sO_{\bP}(1)$.
Then $p_*\sO_{\bP}(1)\simeq\mathcal{L}^{-1}\oplus\mathcal{L}^{-1}\oplus\mathcal{O}_S$, and the last summand gives rise to a section $u\in H^0(\bP,\sO_{\bP}(1))$. Similarly, the first two summands of the isomorphism $p_*(p^*\mathcal{L}\otimes\sO_{\bP}(1))\simeq\mathcal{O}_S\oplus\mathcal{O}_S\oplus\mathcal{L}$
induce two sections $s,t\in H^0(\bP,p^*\mathcal{L}\otimes\sO_{\bP}(1))$.

\subsubsection{Characteristic not \texorpdfstring{$2$}{2}}
\label{parnot2}
Suppose first that $k$ has characteristic different from~$2$.
Define an embedded conic bundle $Y\subset\bP$ over $S$ by the equation
\begin{equation*}
Y:=\{s^2-t^2=u^2F\}\subset\bP.
\end{equation*}
Kummer theory yields a surjection $\kappa:k^*\surj H^1(k,\Z/2\Z)$ with kernel $(k^*)^2$. We fix $\alpha\in k^*$, we set $a:=\kappa(\alpha)$, and we choose $\beta\in\overline{k}$ such that $\beta^2=\alpha$. 
We consider the embedded conic bundle $X_{\alpha}\subset\bP$ over $S$ with equation
\begin{equation}
\label{coniceq}
X_{\alpha}:=\{s^2-\alpha t^2=u^2F\}\subset\bP,
\end{equation}
which will turn out to be a twist of $Y$ (see Proposition \ref{conicbundles} (i)).

\begin{lem}
\label{ratsqrt}
If $S$ is $k(\beta)$-rational, then so is $X_{\alpha}$. 
\end{lem}

\begin{proof}
The generic fiber of the projection $p|_{X_{\alpha}}:X_{\alpha}\to S$ is a conic that has a $k(\beta)(S)$\nobreakdash-point given by $s=\beta$, $t=1$ and $u=0$, hence is $k(\beta)(S)$-rational. The lemma follows at once.
\end{proof}

The $S$\nobreakdash-automorphism $\delta$ of $Y$ given by the formula $(s,t,u)\mapsto (s,-t,u)$ yields a $\Gamma_k$\nobreakdash-equivariant morphism $\psi:\Z/2\Z\to \Aut(Y_{\overline{k}})$ with $\psi(1)=\delta$.
For $a\in H^1(k,\Z/2\Z)$, we consider the twist ${}_{\psi(a)}Y$ of $Y$, where we still denote by $\psi$ the composition $H^1(k,\Z/2\Z)=Z^1(k,\Z/2\Z)\to Z^1(k,\Aut(Y_{\overline{k}}))$.

\begin{prop}
\label{conicbundles}
 Assume that $k$ is perfect. 
\begin{enumerate}[(i)]
\item The varieties $X_{\alpha}$ and ${}_{\psi(a)}Y$ are $k$-isomorphic.
\item There is an isomorphism $J^3Y\simeq J^1C$ of ppavs over $k$.
\item There is an isomorphism $J^3X_{\alpha}\simeq {}_{\varphi(a)}(J^1C)$ of ppavs over $k$.
\item If $C_{\overline{k}}$ is connected, of genus $\geq 2$ and not hyperelliptic, and if $\beta\notin k$, then $X_{\alpha}$ is not $k$-rational.
\end{enumerate}
\end{prop}

\begin{proof}
The isomorphism $i:Y_{\overline{k}}\simeq (X_{\alpha})_{\overline{k}}$ given by $(s,t,u)\mapsto (s,t/\beta,u)$ satisfies $\gamma(i(y))=i(\psi(a_{\gamma})\cdot\gamma(y))$ for all $y\in Y(\overline{k})$ and $\gamma\in\Gamma_k$. The description of ${}_{\psi(a)}Y$ given in \S\ref{twists} then shows that ${}_{\psi(a)}Y\simeq X_{\alpha}$ as varieties over $k$, proving (i).

Define $W:=\bP_S(\sL^{-1}\oplus \sO_S)$ with projection $q:W\to S$ and tautological bundle $\sO_W(1)$, the two factors of $\sL^{-1}\oplus \sO_S$ inducing sections $v\in H^0(W,q^*\sL\otimes\sO_W(1))$ and $w\in H^0(W,\sO_W(1))$. The $S$-rational map $W\dashrightarrow Y$ given by the formula $(v,w)\mapsto(s,t,u)=((w^2F+v^2)/2,(w^2F-v^2)/2,vw)$ identifies $Y$ with the blow-up of $W$ along the curve $\{v=F=0\}$, which is isomorphic to $C$. Let $E:=\{F=s-t=0\}$ be the exceptional divisor, with inclusion $j:E\hookrightarrow Y$.  
Since $W$ and hence also~$Y$ are smooth projective $\overline{k}$-rational threefolds, they satisfy Property \ref{assum} by Corollary~\ref{CGk}.
One has ${\CH^2(W_{\overline{k}})_{\alg}\simeq \CH^1(S_{\overline{k}})_{\alg}=0}$, hence $J^3W=0$, by the computation of the Chow groups of a projective bundle \cite[Theorem 3.3 (b)]{Fulton} and because $S$ is $\overline{k}$\nobreakdash-rational.  We deduce that $J^3Y\simeq J^3W\times J^1C\simeq J^1C$ as ppavs over~$k$ by Lemma~\ref{lembl}, and (ii) is proved.

The inverse image $Z:=(p|_Y)^{-1}(C)\subset Y$ is the union of $E$ and $\delta(E)$.
The total space of the normalization $\nu:\widetilde{Z}\to Z$ is thus isomorphic to the disjoint union of $E$ and $\delta(E)$.  Let $f:C\to S$ and $g:Z\to Y$ be the inclusions, and $h:=p|_Z\circ\nu:\widetilde{Z}\to C$. Applying \cite[Proposition~6.6 (b) and (c)]{Fulton} (especially the statement there concerning \cite[Theorem~6.2~(a)]{Fulton}) shows that $g_*\circ(p|_Z)^*=(p|_Y)^*\circ f_*:\CH^1(C_{\overline{k}})\to \CH^2(Y_{\overline{k}})$. One verifies easily that $(p|_Z)^*=\nu_*\circ h^*:\CH^1(C_{\overline{k}})\to \CH^1(Z_{\overline{k}})$ on the generators of $\CH^1(C_{\overline{k}})$. Consequently, $(g\circ\nu)_*\circ h^*=(p|_Y)^*\circ f_*:\CH^1(C_{\overline{k}})\to \CH^2(Y_{\overline{k}})$. Since $\CH^2(S_{\overline{k}})_{\alg}=0$ as $S$ is $\overline{k}$-rational, the map $(g\circ\nu)_*\circ h^*:\CH^1(C_{\overline{k}})_{\alg}\to \CH^2(Y_{\overline{k}})_{\alg}$ vanishes identically. Equivalently, $$(1+\delta^*)\circ j_*\circ (p|_E)^*:\CH^1(C_{\overline{k}})_{\alg}\to \CH^2(Y_{\overline{k}})_{\alg}$$
is identically zero. Since $ j_*\circ (p|_E)^*:\CH^1(C_{\overline{k}})_{\alg}\to \CH^2(Y_{\overline{k}})_{\alg}$ is an isomorphism by the description of $Y$ as a blow-up of $W$ and by the computation of the Chow groups of blow-ups and projective bundles \cite[Theorem 3.3 (b), Proposition~6.7~(e)]{Fulton}, we see that 
$\delta^*z=-z$ for all $z\in\CH^2(Y_{\overline{k}})_{\alg}$. As a consequence of this identity, one has $\chi\circ\psi=\varphi:\Z/2\Z\to \Aut_{\ppav}(J^3Y_{\overline{k}})$, where we use the notation of \S\S\ref{twists}--\ref{qtwists}.
By (i), Proposition \ref{twisttwist} and (ii),
$J^3X_{\alpha}\simeq J^3({}_{\psi(a)}Y)\simeq {}_{\varphi(a)}(J^3Y)\simeq {}_{\varphi(a)}(J^1C)$.

Finally, one deduces (iv) from (iii), Corollary \ref{CGk} and Proposition \ref{Torelli}.
\end{proof}

\begin{rem}
Over $k=\C$, Mumford described the intermediate Jacobian of a conic bundle as a Prym variety (see \cite[Appendix C]{CG} and \cite[Th\'eor\`eme~2.1]{Beauville}). Our computation that $J^3X_{\alpha}={}_{\varphi(a)}(J^1C)$ in Proposition~\ref{conicbundles}~(iv) is a variant of this result.
Assuming for simplicity that $C$ is geometrically connected of genus~$g$ and $\beta\notin k$, one may think of ${}_{\varphi(a)}(J^1C)$ as playing the role of the Prym variety of the double cover $C_{k(\beta)}\to C$.
Indeed, the Jacobian $J^1(C_{k(\beta)})$
of the smooth projective connected curve $C_{k(\beta)}$ over~$k$
coincides with the Weil restriction of scalars $\Res_{k(\beta)/k}((J^1C)_{k(\beta)})$,
and there is a canonical exact sequence of abelian varieties
\begin{align*}
\xymatrix@C=3.5em{
0 \ar[r] &
{}_{\varphi(a)}(J^1C)
\ar[r] &
\Res_{k(\beta)/k}((J^1C)_{k(\beta)})
\ar[r]^(.66){N_{k(\beta)/k}} &
J^1C \ar[r] & 0
}
\end{align*}
(obtained by twisting the exact sequence $0 \to J^1C \to J^1C \times J^1C \to J^1C \to 0$).
This differs from the classical setting in that the total space $C_{k(\beta)}$ of the double cover is not geometrically connected, which explains that the dimension of our intermediate Jacobian ${}_{\varphi(a)}(J^1C)$ is equal to $g$ and not to~$g-1$.
\end{rem}

\begin{cor}
\label{coroexemple}
Suppose that $\beta\notin k$, and let $F\in H^0(\bP^2_{k},\sO_{\bP^2_k}(2d))$ be the equation of a smooth plane curve for some $d\geq 2$. The smooth projective variety $X_{\alpha}$ over~$k$ with equation $\{s^2-\alpha t^2=u^2F\}$ as in (\ref{coniceq}) is $k(\beta)$-rational but not $k$-rational.
\end{cor}

\begin{proof}
We use the above results with $S=\bP^2_k$ and $\sL=\sO_{\bP^2_k}(d)$.
The first assertion is Lemma \ref{ratsqrt}. To prove the second assertion, one may apply Proposition~\ref{conicbundles}~(iv) over the perfect closure of $k$ since $C_{\overline{k}}$ is non-hyperelliptic of genus~$\geq 2$, as is any smooth plane curve of degree $\geq 4$ ($K_{C_{\overline{k}}}$ is very ample as a positive multiple of~$\sO_{\bP^2_{\overline{k}}}(1)$, whereas the canonical bundle of a hyperelliptic curve is not).
\end{proof}

\begin{rem}
Corollary \ref{coroexemple} would fail for $d=1$, as $\{s^2-\alpha t^2=u^2(x^{2}+y^{2}+z^2)\}$ is birational to the smooth quadric with a $k$-point $\{s^2-\alpha t^2=x^2+y^2+z^2\}\subset\bP^4_k$, hence is $k$\nobreakdash-rational. In this case, $C_{\overline{k}}$ has genus $0$.

Let us illustrate further the importance of the hypothesis that $C_{\overline{k}}$ is of genus $\geq 2$ and not hyperelliptic in Proposition \ref{conicbundles} (iv).
Fix $d\geq 1$, let $\Phi\in H^0(\bP^1_{k},\sO_{\bP^1_{k}}(2d))$ be a polynomial with pairwise distinct roots over $\overline{k}$, and consider the 
 projective bundle $q:S:=\bP_{\bP^1_{k}}(\sO_{\bP^1_{k}}\oplus\sO_{\bP^1_{k}}(d))\to\bP^1_{k}$ with tautological line bundle $\mathcal{L}:=\sO_S(1)$. There are two canonical sections $v\in H^0(S,\sO_S(1))$ and $w\in H^0(S,q^*\sO_{\bP^1_{k}}(-d)\otimes\sO_S(1))$, and one may consider $C:=\{F=0\}\subset S$ with $F:=v^2-\Phi w^2$. The curve $C_{\overline{k}}$ is smooth, connected, of genus $g=d-1$, and hyperelliptic. The conic bundle $X_{\alpha}:=\{s^2-\alpha t^2=u^2F\}$ over $S$ as in (\ref{coniceq}) satisfies $J^3X_{\alpha}\simeq J^1C'$, where $C'$ is a hyperelliptic twist of $C$ (by Proposition \ref{conicbundles}~(iii) and Proposition~\ref{Torelli}). Consequently, one cannot deduce from Corollary \ref{CGk} that $X_{\alpha}$ is not $k$\nobreakdash-rational. 
This is fortunate, because $X_{\alpha}$ is birational to the quadric surface bundle $\{s^2-\alpha t^2-v^2+\Phi w^2=0\}$ over~$\bP^1_{k}$. This quadric bundle has a rational section given by $s=v=1$ and $t=w=0$, showing that $X_{\alpha}$ is $k$-rational.
\end{rem}

\begin{ex}
\label{extexte}
When $d=2$, the varieties of Corollary \ref{coroexemple} are often $k$-unirational. We only give one example: we show that the variety $X_{\alpha}$ over $k$ defined by the equation $\{s^2-\alpha t^2=u^2(x^{4}+y^{4}+z^4)\}$ as in (\ref{coniceq}) is $k$-unirational.
Applying \cite[Lemma~3.5]{dP2} with $a_1=a_2=1$, $a_3=-1$ and $a_4=a_5=a_6=0$ shows that the degree $2$ del Pezzo surface $T:=\{s^2=x^4+y^4+1\}\subset X_{\alpha}$ is $k$-unirational. Then $X_{\alpha}$ is dominated by the fiber product $X_{\alpha}\times_{\bP^2_k}T$ which, as a conic bundle with a rational section over the  $k$-unirational variety $T$, is $k$-unirational.
\end{ex}

\subsubsection{Characteristic \texorpdfstring{$2$}{2}}
Let us now assume that $k$ has characteristic $2$. We only explain how to modify the statements and arguments of \S\ref{parnot2} in this case.

Artin--Schreier theory yields a surjection $\kappa:k\surj H^1(k,\Z/2\Z)$, whose kernel consists of the elements of the form $\beta^2+\beta$ for some $\beta\in k$. We fix $\alpha\in k$, we set $a:=\kappa(\alpha)$, and we choose $\beta\in\overline{k}$ such that $\beta^2+\beta=\alpha$. 

We define an embedded conic bundle $Y:=\{s^2+st=u^2F\}\subset\bP$, and we let $\delta$ be the $S$-automorphism of $Y$ given by the formula $(s,t,u)\mapsto (s+t,t,u)$.
We consider the embedded conic bundle $X_{\alpha}\subset\bP$ over $S$ with equation
\begin{equation}
\label{coniceq2}
X_{\alpha}:=\{s^2+st+\alpha t^2=u^2F\}\subset\bP.
\end{equation}

With these modifications, Lemma \ref{ratsqrt} and Proposition \ref{conicbundles} continue to hold, with the same proofs (in the proof of Proposition \ref{conicbundles} (i), take $i:(s,t,u)\mapsto (s+t\beta,t,u)$).
We deduce from these statements an analogue of Corollary \ref{coroexemple}, using exactly the same arguments.

\begin{cor}
\label{coroexemple2}
Suppose that $\beta\notin k$, and let $F\in H^0(\bP^2_{k},\sO_{\bP^2_k}(2d))$ be the equation of a smooth plane curve for some $d\geq 2$. The smooth projective variety $X_{\alpha}$ over~$k$ with equation $\{s^2+st+\alpha t^2=u^2F\}$ as in (\ref{coniceq2}) is $k(\beta)$-rational but not $k$-rational.
\end{cor}

\begin{ex}
\label{extexte2}
When $d=2$, the varieties of Corollary \ref{coroexemple2} are often $k$-unirational. We only give one example: we show that the variety $X_{\alpha}$ over $k$ defined by the equation $\{s^2+st+\alpha t^2=u^2(x^3y+y^3z+z^3x)\}$ as in (\ref{coniceq}) is $k$-unirational.
Arguing as in Example \ref{extexte}, it suffices to show that $T:=\{s^2=x^3y+y^3z+z^3x\}$ is $k$\nobreakdash-unirational. But this variety is even $\bF_2$-unirational since
$$\bF_2(T)=\bF_2(x,y,(x^3y+y^3+x)^{1/2})\subset\bF_2(x^{1/2},y^{1/2}).$$
\end{ex}

\section{Unramified cohomology of real threefolds}
\label{secunram}

We now restrict to the field $k=\R$ of real numbers and study in detail another strategy to show that a $\C$-rational threefold over~$\R$ is not $\R$-rational, making use of unramified cohomology.
Recall that $G=\Gal(\C/\R)$.

\subsection{Bloch--Ogus theory}
\label{unramcoho}
If $X$ is a smooth variety over $\R$, the group $G$ acts continuously on $X(\C)$ and we will consider, for any $G$\nobreakdash-module $M$ and any $i\geq 0$, the $G$-equivariant Betti cohomology groups $H^i_G(X(\C),M)$. Let $\sH^i_X(M)$ be the Zariski sheaf on~$X$ associated with the presheaf $U\mapsto H^i_G(U(\C),M)$. The degree $i$ \textit{unramified cohomology group} of $X$ with coefficients in $M$ is $H^i_{\nr}(X,M):=H^0(X,\sH^i_X(M))$.
The sheaf $\sH^i_{\Spec(\R)}(M)$ is the constant sheaf $H^i(G,M)$, and pulling-back along the structural morphism yields a morphism $H^i(G,M)\to H^i_{\nr}(X,M)$.

We refer to \cite[\S 5.1]{BW} for more information on the sheaves $\sH^i_X(M)$ and their Zariski cohomology groups. It is explained there that
\cite[Corollary 5.1.11]{BOG} may be applied in this context, which is usually referred to as the validity of Gersten's conjecture. In particular, defining
$$\mathcal{C}^{i,c}_X(M):=\bigoplus_{x\in X^{(c)}}\iota_{x,*}\varinjlim_{U\subset \overline{\{x\}}} H^{i-c}_G(U(\C),M(-c)),$$
where $X^{(c)}$ is the set of codimension $c$ points of $X$, $\iota_x:x\to X$ is the inclusion, and $U$ runs  over the dense open subvarieties of $\overline{\{x\}}$, the sheaf $\sH^i_X(M)$ admits a flasque resolution by a Cousin complex
\begin{equation}
\label{Gersten}
0\to \sH^i_X(M)\to \mathcal{C}^{i,0}_X(M)\to \mathcal{C}^{i,1}_X(M)\to\dots
\end{equation}
whose arrows are induced by residue maps in long exact sequences of $G$-equivariant cohomology with support.
Consequently, the Zariski cohomology groups of $\sH^i_X(M)$ may be computed as the cohomology of the complex obtained by taking the global sections of the Cousin complex.
Using this description, the arguments of \cite[Appendice A]{CTV} adapt to the real setting and show that correspondences between smooth projective varieties act naturally on the groups $H^j(X,\sH^i_X(M))$.

\subsection{Obstructions to rationality} 
\label{obstrat}
We first recall two definitions originating respectively from \cite[Definition 3.1, Lemma 3.5]{Saltman} 
and \cite[\S 1.2]{CH0trivial}.

\begin{Def}
\label{retract}
A smooth projective variety $X$ over a field $k$ is \textit{retract} $k$-\textit{rational} if there exist a dense open subset $U\subset X$, a $k$-rational variety $V$ and morphisms $f:U\to V$ and $g:V\to U$ such that $g\circ f=\Id$.
It is \textit{universally} $\CH_0$-\textit{trivial} if for every field extension $k\subset l$, the degree map $\deg:\CH_0(X_l)\to\Z$ is an isomorphism.
\end{Def}

It is obvious that a smooth projective $k$-rational variety is retract $k$-rational (more generally, stably $k$-rational varieties are retract $k$-rational), and a smooth projective retract $k$-rational variety is universally $\CH_0$-trivial by \cite[Lemme 1.5]{CtPi}.
 The following proposition is a variant of classical results (see for instance \cite[Corollaire du Th\'eor\`eme 2]{BVbirationnels} or \cite[Theorem 1.4]{CH0trivial}).

\begin{prop}
\label{trivialunramcoho}
Let $X$ be a smooth projective variety over $\R$ that is universally $\CH_0$-trivial. Then for any $i\geq 0$ and any $G$-module $M$, $H^i(G,M)\stackrel{\sim}\longrightarrow H^i_{\nr}(X,M)$.

In particular, the conclusion holds if $X$ is retract $\R$-rational.
\end{prop}

\begin{proof}
Since $X$ has a zero-cycle of degree $1$, it has a real point $x\in X(\R)$. The restriction to $x$ is a retraction of $H^i(G,M)\to H^i_{\nr}(X,M)$, showing its injectivity.

By \cite[Lemma 1.3]{CH0trivial}, $X$ admits an integral decomposition of the diagonal: if $d$ is the dimension of $X$, there is an equality $\Delta_X=\{x\}\times X+Z\in \CH^d(X\times X)$, where $\Delta_X$ is the diagonal and $Z$ is supported on a closed subset $X\times D$ where $D\varsubsetneq X$. One may now argue as in the proof of \cite[Proposition 3.3 (i)]{CTV} by letting these correspondences act on $\alpha\in H^i_{\nr}(X,M)$. Of course $\Delta_{X,*}\alpha=\alpha$, and $(X\times\{x\})_*\alpha$ is in the image of $H^i(G,M)\to H^i_{\nr}(X,M)$. Moreover, $Z_*\alpha$ vanishes in the complement of $D$, hence vanishes as one sees immediately from the description of $H^i_{\nr}(X,M)$ as a cohomology group of the complex of global sections of the Cousin complex. We have shown the surjectivity of $H^i(G,M)\to H^i_{\nr}(X,M)$.

Finally, the last assertion follows from \cite[Lemme 1.5]{CtPi}.
\end{proof}

\subsection{The case of \texorpdfstring{$\C$-rational}{𝐂-rational} threefolds}
We understand completely when these invariants allow one to show that a $\C$-rational threefold is not (retract) $\R$-rational. 

\begin{thm}
\label{thunr3}
Let $X$ be a smooth projective threefold over $\R$ that is $\C$-rational. The following are equivalent:
\begin{enumerate}
\item For any $i\geq 0$ and any $G$-module $M$, $H^i(G,M)\stackrel{\sim}\longrightarrow H^i_{\nr}(X,M)$.
\item The variety $X$ satisfies:
\begin{enumerate}[(i)]
\item $X(\R)$ has exactly one connected component.
\item $\Pic(X_{\C})$ is a permutation $G$-module.
\item The cycle class map $\cl_{\R}:\CH_1(X)\to H_1(X(\R),\Z/2\Z)$ is surjective.
\end{enumerate}
\end{enumerate}
\end{thm}

Combining Theorem~\ref{thunr3} with Proposition \ref{trivialunramcoho}, we see that (i), (ii) and (iii) are necessary conditions for the (retract) $\R$-rationality of $X$. These conditions have already been explained and discussed in the introduction. We only recall here that a $G$-module $M$ is a \textit{permutation} $G$-\textit{module} if it is isomorphic to a direct sum of copies of  $\Z$ and $\Z[G]$. We will use several times the following lemma.

\begin{lem}
\label{permulemme}
A finitely generated torsion-free $G$-module $M$ is a permutation $G$\nobreakdash-module if and only if $H^1(G,M)=0$.
\end{lem}

\begin{proof}
 By \cite[I (3.5.1)]{Silhol}, a finitely generated torsion-free $G$-module is a direct sum of $G$-modules isomorphic to $\Z$, $\Z(1)$ and $\Z[G]$. The lemma follows.
\end{proof}

\begin{proof}[Proof of Theorem \ref{thunr3}]
The variety $X_{\C}$ is rational, hence connected, so that all open subsets $U\subset X$ are geometrically connected. Consequently, $\sH^0_X(M)$ is the constant sheaf $H^0(G,M)$, and $H^0(G,M)\stackrel{\sim}\longrightarrow H^0_{\nr}(X,M)$.

The group $H^1(X(\C),M)$ is a birational invariant of smooth projective complex varieties. Since $X_{\C}$ is rational, it vanishes. As a consequence, the Hochschild--Serre spectral sequence \cite[(1.4)]{BW} provides an isomorphism $H^1(G,M)\stackrel{\sim}\longrightarrow H^1_G(X(\C),M)$. Combining it with the isomorphism $H^1_G(X(\C),M)\stackrel{\sim}\longrightarrow H^1_{\nr}(X,M)$ given by the coniveau spectral sequence \cite[(5.1), (5.2)]{BW} shows that $H^1(G,M)\stackrel{\sim}\longrightarrow H^1_{\nr}(X,M)$.

View $H^i(G,M)$ as a constant sheaf on $X(\R)$ for the euclidean topology and let ${\iota:X(\R)\to X}$ be the inclusion. The natural restriction map $\sH^i_X(M)\to\iota_*H^i(G,M)$ is an isomorphism if $i\geq 4$ by \cite[Proposition 5.1 (iv)]{BW}. It follows that the restriction to real points induces an isomorphism $H^i_{\nr}(X,M)\stackrel{\sim}\longrightarrow H^0(X(\R),H^i(G,M))$ for all $i\geq 4$. We deduce that if $X(\R)$ has exactly one connected component, the pull-back 
$H^i(G,M)\to H^i_{\nr}(X,M)$ is an isomorphism for any $G$-module $M$ and any $i\geq 4$, and taking $M=\Z/2\Z$ shows that the converse holds. 
 
From now on, we may assume that $X(\R)$ has exactly one connected component. In particular, the morphisms $H^i_{\nr}(X,M)\stackrel{\sim}\longrightarrow H^0(X(\R),H^i(G,M))=H^i(G,M)$ induced by restrictions to real points are retractions of the pull-back morphisms $H^i(G,M)\to H^i_{\nr}(X,M)$, showing that the latter are injective. Statement~(1) is thus tantamount to their surjectivity, or equivalently to the vanishing of
$$H^i_{\nr}(X,M)_0:=\{\alpha\in H^i_{\nr}(X,M)\mid \alpha|_{x}=0\textrm{ for all }x\in X(\R)\}.$$

Let us complete the proof that (1) implies (2): it remains to prove (ii) and~(iii). By \cite[Proposition 4.2.3 (a)]{CT} and comparison between equivariant Betti cohomology and \'etale cohomology (see \cite{Cox}, \cite[Corollary 15.3.1]{Scheiderer}), there is an isomorphism 
\begin{equation}
\label{Brnr}
\Br(X)=H^2_{\nr}(X,\Q/\Z(1)).
\end{equation}
 We deduce from (1) that $\Br(\R)\to\Br(X)$ is an isomorphism. Since $H^3(G,\C^*)=0$, the Hochschild--Serre spectral sequence 
\begin{equation}
\label{HochschildSerre}
E_2^{p,q}=H^p(G,H^q_{\et}(X_{\C},\mathbb{G}_m))\implies H^{p+q}_{\et}(X,\mathbb{G}_m)
\end{equation}
now implies that $H^1(G,\Pic(X_{\C}))=0$.
Since $X_{\C}$ is rational, $\Pic(X_{\C})$ is torsion-free and finitely generated, and Lemma \ref{permulemme} shows that $\Pic(X_{\C})$ is a permutation $G$-module, proving (ii).
By \cite[Remark 5.3 (iii)]{BW}, the vanishing of 
$H^3_{\nr}(X,\Q/\Z(2))_0$ implies that the real integral Hodge conjecture for $1$-cycles on $X$ (see \cite[Definition~2.2]{BW}) holds. In turn, this implies (iii) by \cite[Theorem 3.22]{BW}. 

It remains to prove that (2) implies (1) in degrees $i=2,3$.  Writing $M$ as the direct limit of its finitely generated sub-$G$-modules, and using the fact that sheafification and taking cohomology commute with colimits, we may assume that it is finitely generated. 

Let us first deal with $i=3$. Define $B$ to be the direct sum of one copy of $\Z[G]$ (resp.\ of $\Z$) for each element in a finite generating subset of $M$ (resp.\ of $M^G$). It is a finitely generated permutation $G$-module. By construction, the natural $G$\nobreakdash-equivariant morphism $p:B\to M$ is such that both $p$ and $p|_{B^G}:B^G\to M^G$ are surjective.  Let $A$ be the kernel of $p$. The long exact sequence of group cohomology associated with $0\to A\to B\xrightarrow{p} M\to 0$ shows that $H^1(G,A)=0$ and Lemma \ref{permulemme} implies that $A$ is a permutation $G$-module. 

Taking long exact sequences of $G$-equivariant cohomology associated with the short exact sequence $0\to A\to B\to M\to 0$ on Zariski open subsets $U\subset X$ and sheafifying gives rise to a long exact sequence of Zariski sheaves on $X$:
$$\sH^2_X(B)\to\sH^2_X(M)\to\sH^3_X(A)\to\sH^3_X(B)\to\sH^3_X(M)\to\sH^4_X(A)\to\sH^4_X(B).$$
 Since $B\otimes_{\Z}\Q\to M\otimes_{\Z}\Q$ has a $G$-equivariant section, $\sH^2_X(B\otimes_{\Z}\Q)\to\sH^2_X(M\otimes_{\Z}\Q)$ is surjective. It follows that the cokernel of $\sH^2_X(B)\to\sH^2_X(M)$ is a torsion sheaf. Since $\sH^3_X(A)$ has no torsion by \cite[Proposition 5.1 (ii)]{BW}, we deduce that $\sH^3_X(A)\to\sH^3_X(B)$ is injective:
\begin{equation}
\label{ABM}
0\to\sH^3_X(A)\to\sH^3_X(B)\to\sH^3_X(M)\to\sH^4_X(A)\to\sH^4_X(B).
\end{equation}

Since $X_{\C}$ is rational, the groups
$H^1(X,\sH^3_X(\Z[G]))=H^1(X_{\C},\sH^3_{X_{\C}}(\Z))$
and $H^0(X,\sH^3_X(\Z[G]))= H^0(X_{\C},\sH^3_{X_{\C}}(\Z))$
(see \cite[Proposition 5.1 (i)]{BW}) both vanish,
by \cite[Proposition~3.3~(iii), Proposition~3.4]{CTV}.
As $H^0(X,\sH^3_X(\Z))$ is a subgroup of the latter by \cite[Proposition 5.1 (i), (iii)]{BW},
it also vanishes.
In addition, combining \cite[Theorem 3.22]{BW} and  \cite[(5.9)]{BW} shows that (iii) implies the vanishing of $H^1(X,\sH^3_X(\Z))$.
All in all, we have proved that
\begin{equation}
\label{vanishings}
H^1(X,\sH^3_X(A))=H^0(X,\sH^3_X(B))=0.
\end{equation}
There is a commutative diagram whose first row is exact, whose second row is a complex obtained by taking global sections in (\ref{ABM}), and in which two vertical arrows are isomorphisms by the case $i=4$ already dealt with:
\begin{align*}
\xymatrix
@R=0.4cm
{
H^3(G,M)\ar^{}[r]\ar[d]^{} & H^4(G,A)\ar[r]^{}\ar[d]^{\wr} &H^4(G,B)\ar[d]^{\wr}  \\
H^3_{\nr}(X,M)\ar^{}[r]& H^4_{\nr}(X,A)\ar[r]^{}&H^4_{\nr}(X,B).
}
\end{align*}
The exactness of (\ref{ABM}) and the vanishings (\ref{vanishings}) imply that $H^3_{\nr}(X,M)\to H^4_{\nr}(X,A)$ is injective. A diagram chase then shows that $H^3(G,M)\to H^3_{\nr}(X,M)$ is surjective, which is what we needed to prove.

It remains to settle the case $i=2$ when $M$ is finitely generated. Applying the above arguments to $M(1)$ instead of $M$, we find a short exact sequence of $G$-modules $0\to C\to D\to M\to 0$, where $C$ and $D$ are finite direct sums of $G$-modules isomorphic to $\Z(1)$ or $\Z[G]$, giving rise to a long exact sequence:
\begin{equation}
\label{CDM}
0\to\sH^2_X(C)\to\sH^2_X(D)\to\sH^2_X(M)\to\sH^3_X(C)\to\sH^3_X(D).
\end{equation}
The group $H^0(X,\sH^2_X(\Z[G]))= H^0(X_{\C},\sH^2_{X_{\C}}(\Z))$ (see \cite[Proposition 5.1 (i)]{BW}) vanishes by \cite[Proposition 3.3 (i)]{CTV} because $X_{\C}$ is rational. Since $H^0(X,\sH^2_X(\Z(1)))$ is a subgroup of it by \cite[Proposition 5.1 (iii)]{BW}, this group also vanishes. We deduce:
\begin{equation}
\label{vanishing}
H^2_{\nr}(X,D)=0.
\end{equation}
Both natural morphisms $(\Z[G]\otimes_{\Z}D)^G\to (\Z[G]\otimes_{\Z}M)^G$ and $D(1)^G\to M(1)^G$ are surjective, the first one because it can be identified with $D\to M$, and the second one because $H^1(G,C(1))=0$.
Since $X_{\C}$ is rational, $\Pic(X_{\C})\simeq H^2(X(\C),\Z(1))$, which is a permutation $G$-module by (ii). It follows that $[H^2(X(\C),\Z)\otimes D]^G\to[H^2(X(\C),\Z)\otimes M]^G$ is surjective. Since $X_{\C}$ is rational, its Artin--Mumford invariant $H^3(X(\C),\Z)_{\tors}$ vanishes \cite[Proposition 1]{AM} and we deduce from the universal coefficient theorem \cite[5.5 Theorem 10]{Spanier}
that $H^2(X(\C),D)^G\to H^2(X(\C),M)^G$ is surjective. For a $G$-module $N$, let us consider the Hochschild--Serre spectral sequence \cite[(1.4)]{BW}:
 $$E_2^{p,q}=H^p(G,H^q(X(\C),N))\implies H^{p+q}_G(X(\C),N).$$
We have seen above that $H^1(X(\C),N)=0$, and restricting to a real point shows that the edge maps $H^i(G,N)\to H^{i}_G(X(\C),N)$ are injective. Applying this to $N=D$ and $N=M$ gives rise to a commutative diagram with exact row:
\begin{align*}
\xymatrix
@R=0.4cm
{
& & H^2_{G}(X(\C),D)\ar[r]^{\sim}\ar[d]^{} &H^2(X(\C),D)^G\ar@{->>}[d] & \\
0\ar[r]&H^2(G,M)\ar^{}[r]& H^2_{G}(X(\C),M)\ar[r]^{}&H^2(X(\C),M)^G\ar[r]& 0.
}
\end{align*}
Since $X(\R)$ is connected, we deduce from the diagram above an isomorphism 
$H^2_G(X(\C),M)_0\stackrel{\sim}\longrightarrow H^2(X(\C),M)^G$, where we set, for all $G$-modules $N$ and $i\geq 0$:
$$H^i_{G}(X(\C),N)_0:=\{\alpha\in H^i_{G}(X(\C),N)\mid \alpha|_{x}=0\textrm{ for all }x\in X(\R)\}.$$
The image of $H^2_{G}(X(\C),D)$ in $H^2_{G}(X(\C),M)$ is contained in $H^2_G(X(\C),M)_0$, and we get a surjection 
$H^2_{G}(X(\C),D)\twoheadrightarrow H^2_G(X(\C),M)_0$.
From the long exact sequence of equivariant cohomology, we deduce an injection $H^3_G(X(\C),C)_0\hookrightarrow H^3_G(X(\C),D)_0$. 
The coniveau spectral sequence \cite[(5.1), (5.2)]{BW} yields, for any $G$\nobreakdash-module~$N$, an injection $H^1(X,\sH^2_X(N))\hookrightarrow H^3_G(X(\C),N)$ whose image has coniveau $\geq 1$, hence belongs to $H^3_G(X(\C),N)_0$ (indeed, since $X$ has a smooth $\R$-point, the implicit function theorem shows that $X(\R)$ is Zariski dense in $X$). Applying it to $N=C$ and $N=D$, we get an injection:
\begin{equation}
\label{injection}
H^1(X,\sH^2_X(C))\hookrightarrow H^1(X,\sH^2_X(D)).
\end{equation}

There is a commutative diagram whose first row is exact, whose second row is a complex obtained by taking global sections in (\ref{CDM}), and in which two vertical arrows are isomorphisms by the case $i=3$ already dealt with:
\begin{align*}
\xymatrix
@R=0.4cm
{
H^2(G,M)\ar^{}[r]\ar[d]^{} & H^3(G,C)\ar[r]^{}\ar[d]^{\wr} &H^3(G,D)\ar[d]^{\wr}  \\
H^2_{\nr}(X,M)\ar^{}[r]& H^3_{\nr}(X,C)\ar[r]^{}&H^3_{\nr}(X,D).
}
\end{align*}
The exactness of (\ref{CDM}), the vanishing (\ref{vanishing}) and the injectivity of (\ref{injection}) imply that $H^2_{\nr}(X,M)\to H^3_{\nr}(X,C)$ is injective. A diagram chase now shows that $H^2(G,M)\to H^2_{\nr}(X,M)$ is surjective, which is what we needed to prove.
\end{proof}

\section{Examples of real threefolds}
\label{secexR}

We finally combine the results of the previous sections to study in detail interesting examples of real threefolds that are $\C$-rational but not $\R$-rational.

In \S\S\ref{reallocus}--\ref{subtrivial}, we consider a variety $X$ defined as in (\ref{coniceq}) with $k=\R$, $\alpha=-1$,  $S=\bP^2_{\R}$ and $\sL=\sO_{\bP^2_{\R}}(d)$ for some $d\geq 1$. It has equation $X=\{s^2+t^2=u^2F\}$ for a homogeneous polynomial $F\in H^0(\bP^2_{\R},\sO_{\bP^2_{\R}}(2d))=\R[x,y,z]_{2d}$ defining a smooth plane curve $C:=\{F=0\}\subset\bP^2_{\R}$. We let $p|_X:X\to \bP^2_{\R}$ be the projection.

\subsection{Set of real points}
\label{reallocus}

It is easy to find such examples for which $X(\R)$ is diffeomorphic to the real locus of a smooth projective $\R$-rational variety.

\begin{prop}
\label{propreallocus}
If $F$ is positive 
on $\R^3\setminus\{0\}$, then $X(\R)$ is diffeomorphic to the real locus of a smooth projective $\R$-rational variety:
\begin{enumerate}[(i)]
\item $X(\R)\simeq \bS^1\times\bP^2(\R)$, if $d$ is even.
\item $X(\R)\simeq(\bS^1\times\bS^2)/(\Z/2\Z)$, where $\Z/2\Z$ acts diagonally by the antipodal involution on both factors, if $d$ is odd.
\end{enumerate}
\end{prop}

\begin{proof}
Let $\mu:\bS^2\to \bP^2(\R)$ be the double cover with Galois group $\Z/2\Z=\{1,\varphi\}$, where $\varphi$ is the antipodal involution of $\bS^2$. Let $\mathbb{L}$ be the $\ci$ real line bundle on $\bP^2(\R)$ associated with $\mathcal{L}=\sO_{\bP^2_{\R}}(d)$. Since $F>0$ and $\bS^2$ is simply connected, there exists a section $G\in H^0(\bS^2,\mu^*\mathbb{L})$ such that $G^2=\mu^*F$, unique up to a sign. Hence $\phi^*G=\varepsilon G$ for some sign $\varepsilon=\pm 1$. Since $\mathbb{L}$ is trivial if and only if $d$ is even,  $\varepsilon=(-1)^d$.

Using the identifications $\bS^1=\{a^2+b^2=1\}$ and $\bS^2=\{x^2+y^2+z^2=1\}$, we deduce that the map $\bS^1\times\bS^2\to X(\R)$  induced by $s=aG(x,y,z)$, $t=bG(x,y,z)$, $u=1$ realizes $X(\R)$ as a quotient of $\bS^1\times\bS^2$ by a diagonal action of $\Z/2\Z$: via the antipodal involution on $\bS^2$ and multiplication by $(-1)^d$ on $\bS^1$.

When $d$ is even, one gets $X(\R)\simeq(\bP^1\times\bP^2)(\R)$. Applying the construction to $d=1$ and $F=x^2+y^2+z^2$ shows that the diagonal quotient $(\bS^1\times\bS^2)/(\Z/2\Z)$ appearing when $d$ is odd is diffeomorphic to the real locus of the smooth projective variety $\{s^2+t^2=u^2(x^2+y^2+z^2)\}$, which is birational to the smooth affine quadric with an $\R$-point $\{s^2+t^2=x^2+y^2+1\}$, hence is $\R$-rational.
\end{proof}

\subsection{Unirationality} Some of the examples we consider are also $\R$-unirational:

\begin{prop}
\label{propunirat}
Suppose that $d=2$. Then $X$ is $\R$-unirational if and only if $F$ is not negative definite on $\R^3\setminus\{0\}$.
\end{prop}

\begin{proof}
If $F$ is negative definite on $\R^3\setminus\{0\}$, then $X(\R)=\varnothing$ so that $X$ cannot be $\R$\nobreakdash-unirational. 

Otherwise, consider the surface $T\subset X$ defined by $\{s=0\}$. The equation $T=\{t^2=u^2F(x,y,z)\}$ shows that it is a smooth degree $2$ del Pezzo surface with a real point. By the implicit function theorem, the real points of $T$ are actually Zariski dense in $T$. It follows from the work of Manin \cite[Theorem 29.4]{Manin} that $T$ is $\R$-unirational. Base changing the conic bundle $X\to \bP^2_{\R}$ by the projection $T\to \bP^2_{\R}$ and base changing it further by a unirational parametrization of $T$, one obtains a conic bundle over an $\R$\nobreakdash-rational surface with a rational section, that is, an $\R$\nobreakdash-rational threefold dominating~$X$. This shows that $X$ is $\R$-unirational.
\end{proof}

\subsection{Examples with trivial unramified cohomology}
\label{subtrivial}

It is also not hard to decide when $X$ has trivial unramified cohomology in the sense of Theorem \ref{thunr3} (1).

\begin{prop}
\label{trivialunrcoho}
The variety $X$ has the property that $H^i(G,M)\stackrel{\sim}\longrightarrow H^i_{\nr}(X,M)$
for any $i\geq 0$ and any $G$-module $M$ if and only if $X(\R)$ is connected and non-empty.
\end{prop}

\begin{proof}
It suffices to show that $X$ satisfies conditions (ii) and (iii) of Theorem \ref{thunr3}~(2) if $X(\R)$ is non-empty.
That condition (iii) holds may be obtained as a combination of \cite[Theorem~6.1]{BW2}  and \cite[Theorem 3.22]{BW}.

To verify~(ii), one may argue as in \cite[Proof of Proposition~2.1]{ctsolidepoonen} (taking $B=\bP^2_\R$ in \emph{loc.\ cit.}).
Alternatively, recall from the proof of Proposition \ref{conicbundles} (iii) that one may write $p|_{X_{\C}}:X_{\C}\to \bP^2_{\C}$ as the composition of the blow-up of a smooth connected curve $X_{\C}\to W$ with exceptional divisor $E=\{F=s-t\sqrt{-1}=0\}\subset X_{\C}$ and of a $\bP^1$\nobreakdash-bundle $W\to \bP^1_{\C}$. We deduce that $\Pic(X_{\C})$ has rank $3$ and is generated by $\sO_{\bP^2_{\C}}(1)$, by $E$ and by any line bundle that has degree one on the generic fiber of $p|_{X_{\C}}:X_{\C}\to \bP^2_{\C}$. If $\sigma(E)=\{F=s+t\sqrt{-1}=0\}\subset X_{\C}$ is the image of $E$ by the complex conjugation $\sigma$, then $E\cup \sigma(E)=(p|_{X_{\C}})^{-1}(C_{\C})$. Consequently, the subgroup $\langle \sO_{\bP^2_{\C}}(1), E\rangle\subset \Pic(X_{\C})$ is $G$-stable and one computes that it is isomorphic to the $G$-module $\Z\oplus\Z(1)$. We have obtained a short exact sequence of $G$-modules:
$$0\to \Z\oplus\Z(1)\to\Pic(X_{\C})\to\Z\to 0,$$
where the projection $\Pic(X_{\C})\to\Z$ computes the degree on the generic fiber of $p_{\C}:X_{\C}\to\bP^2_{\C}$. Since $X(\R)\neq\varnothing$, one has $\Pic(X_{\C})^G=\Pic(X)$ by \cite[8.1/4]{BLR},
and the long exact sequence of $G$-cohomology yields:
\begin{equation}
\label{Picext}
0\to\Z\to\Pic(X)\to \Z\to \Z/2\Z\to H^1(G,\Pic(X_{\C}))\to 0.
\end{equation}
The generic fiber of $p|_X$ is a non-trivial conic as it has non-trivial ramification above~$C$. Consequently, there is no line bundle on $X$ that has degree $1$ on the generic fiber of $p|_X$. We deduce from (\ref{Picext}) that $H^1(G,\Pic(X_{\C}))=0$, hence that $\Pic(X_{\C})$ is a permutation $G$-module by Lemma \ref{permulemme}.
\end{proof}

Combining the results obtained so far, we get:

\begin{thm}
\label{premierexemple}
There exists a smooth projective threefold $X$ over $\R$ that is not $\R$-rational, but is $\C$-rational, $\R$-unirational, and is such that $X(\R)$ is diffeomorphic to $(\bP^1\times\bP^2)(\R)$ and that for any $G$-module $M$ and $i\geq 0$,  $H^i(G,M)\stackrel{\sim}\longrightarrow H^i_{\nr}(X,M)$.
\end{thm}

\begin{proof}
Let $F(x,y,z)\in H^0(\bP^2_{\R},\sO(4))$ be a polynomial that is positive on $\R^3\setminus\{0\}$, and that defines a smooth plane curve (one may take $F(x,y,z)=x^4+y^4+z^4$). The smooth projective variety defined by the equation $X:=\{s^2+t^2=u^2F\}$ as in (\ref{coniceq}) 
has the required properties by Corollary \ref{coroexemple} and Propositions \ref{propreallocus}, \ref{propunirat} and~\ref{trivialunrcoho}.
\end{proof}

\begin{rem}
The variety defined by the equation $X:=\{s^2+t^2=u^2F(x,y,z)\}$ with $F(x,y,z)=x^4-y^4-z^4$ also satisfies the requirements of Theorem~\ref{premierexemple} (except that its real locus is diffeomorphic to the sphere $\bS^3$, hence to the real locus of an $\R$-rational quadric), by Corollary \ref{coroexemple} and Propositions \ref{propunirat} and \ref{trivialunrcoho}.
In this particular example, some arguments may be simplified. In the proof of Proposition \ref{propunirat},  $T(\R)$ is diffeomorphic to a sphere $\bS^2$, and Comessatti's theorem (\cite[pp.\ 54-55]{Comessatti}, see \cite[VI Corollary 6.5]{Silhol}) shows at once that $T$ is $\R$-rational. In the proof of Proposition \ref{trivialunrcoho}, the verification of condition (iii) is immediate as $H_1(X(\R),\Z/2\Z)=H_1(\bS^3,\Z/2\Z)=0$.
\end{rem}

\begin{rem}
\label{remAMnon}
In Theorem \ref{premierexemple}, the assertion that $H^i(G,M)\stackrel{\sim}\longrightarrow H^i_{\nr}(X,M)$ shows that it is not possible to prove that $X$ is not $\R$-rational using Proposition \ref{trivialunramcoho}.
We do not know if $X$ is retract $\R$-rational, or stably $\R$-rational, or if it is universally $\CH_0$-trivial. 
\end{rem}

\subsection{Examples with non-trivial unramified cohomology}
\label{subnontrivial}

To contrast with Theorem \ref{premierexemple}, we give an example of a smooth projective $\C$-rational threefold $X$ over~$\R$ 
that may be proved not to be $\R$-rational using Proposition \ref{trivialunramcoho}, but not using Corollary \ref{CGk}.
Since examples failing condition (i) of Theorem \ref{thunr3} are classical, we restrict to varieties whose real locus is non-empty and connected. In view of the discussion below the statement of Theorem \ref{thunr3intro}, it is not expected that there are such examples for which~(iii) fails. Consequently, we focus on condition (ii). 

\begin{thm}
\label{secondexemple}
There exists a smooth projective threefold $X$ over $\R$ that is not retract $\R$-rational, but is $\C$-rational, $\R$-unirational, whose real locus is diffeomorphic to that of a smooth projective $\R$-rational variety,
and such that $J^3X=0$.
\end{thm}

\begin{proof}
Consider $\Sigma:=\{x^2+y^2=(t-t^3)z^2 \}\subset\bP^2_{\R}\times\bA^1_{\R}$, where $[x:y:z]$ are homogeneous coordinates on $\bP^2_{\R}$ and $t$ is the coordinate on $\bA^1_{\R}$. Let $S$ be a smooth compactification of $\Sigma$ such that the projection to $\bA^1_{\R}$ extends to a relatively minimal conic bundle $\pi: S\to\bP^1_{\R}$. The real locus $S(\R)$ is a disjoint union of two spheres: let $K\subset S(\R)$ be the one for which $t\in[0,1]$. 

Let $\alpha \in \Br(\R(S))$ be the class of the quaternion algebra $(-1,t-t^2)$.
As there exists $D \in \mathrm{Div}(S_\C)$ such that $\mathrm{div}(t-t^2)=N_{\C/\R}(D)$,
the conic over~$\R(S)$ defined by this quaternion algebra extends to a
smooth and projective morphism $f:X \to S$ all of whose fibers are conics.
Indeed, letting $\pi:S_\C \to S$ be the natural morphism and $p:\bP(\sE)\to S$ be the projective bundle
associated with $\sE=\sO_S\oplus \pi_*\sO_{S_\C}(-D)$, the global section $(-1)\oplus (t-t^2)$ of
$\sO_S \oplus \sO_S(-N_{\C/\R}(D)) \subset \Sym^2_{\sO_S} \sE = p_*\sO_{\bP(\sE)}(2)$
defines a global section of $\sO_{\bP(\sE)}(2)$ whose zero locus in $\bP(\sE)$ is the sought for~$X$.
We note that $f(X(\R))=K$.

As $S_{\C}$ is rational and as $\alpha|_{S_{\C}}=0$, the threefold $X_{\C}$ is rational.

The Brauer group of a $\C$\nobreakdash-rational smooth proper surface over~$\R$ whose real locus has $s\geq 1$ connected
components is isomorphic to $(\Z/2\Z)^{2s-1}$
(by \cite[Th\'eor\`eme~4]{Silholrat} and the Hochschild--Serre spectral sequence (\ref{HochschildSerre})). Thus $\Br(S) \simeq (\Z/2\Z)^3$.
Since $S$ is regular, $\Br(S)\to\Br(\R(S))$ is injective by \cite[Corollaire 1.8]{GdB2}, and $X_{\R(S)}$ being a non-trivial conic over $\R(S)$, the kernel of $\Br(\R(S))\to\Br(X_{\R(S)})$ has cardinality exactly $2$ \cite[Proposition 1.5]{CToj}.
It follows that the kernel of $\Br(S)\to\Br(X)$ has cardinality at most $2$, hence that $\Br(X)$ contains a subgroup isomorphic to $(\Z/2\Z)^2$. We deduce that $\Br(\R)\to\Br(X)$ is not an isomorphism. Equivalently, in view of (\ref{Brnr}), $H^2(G,\Q/\Z(1))\to H^2_{\nr}(X,\Q/\Z(1))$ is not an isomorphism.
By Proposition~\ref{trivialunramcoho}, $X$ is not retract $\R$-rational.

Let $g:\bP^1_{\R}\to\bP^1_{\R}$ be a morphism such that $g(\bP^1(\R))\subset\mkern-1mu(0,1)$, and $T$ be a resolution of singularities of the base change of $\pi: S\to\bP^1_{\R}$ by $g$.  By \cite[VI Proposition 3.2, Lemma 3.3]{Silhol}, $T$ is $\R$-rational.
Consider the base change $X_T\to T$ of $f$ by $T\to S$. By \cite[Corollaire 7.5]{GdB3} and since $\Br(\R)\stackrel{\sim}\longrightarrow\Br(\bP^2_{\R})$, the pull-back $\alpha_T\in\Br(T)$ of $\alpha$ on $T$  comes from $\Br(\R)$. By the choice of $\alpha$, this class is trivial in restriction to the real points of $T$, hence is trivial. Consequently, $X_T$ is the projectivization of a rank $2$ vector bundle over $T$, hence is $\R$-rational. We have shown that $X$ is $\R$-unirational.

The Leray spectral sequence of $X_{\C}\to S_{\C}$ shows that $H^3(X(\C),\Z)=0$. By comparison with $\ell$-adic cohomology and Proposition \ref{decodiag}, it follows that $J^3X=0$.

It remains to control $X(\R)$. It follows from the equation of $\Sigma$ and the explicit expression of $\alpha$ that a neighbourhood of $X$ along $X(\R)$ may have been chosen to be the blow-up of the affine variety $\{x^2+y^2=(t-t^3), u^2+v^2=(t-t^2)\}\subset\bA^5_{\R}$ at its two singular points $\{x=y=u=v=t=0\}$ and $\{x=y=u=v=t-1=0\}$. With this choice, the substitutions $x\mapsto \frac{x}{\sqrt{1+t}}$ and $y\mapsto \frac{y}{\sqrt{1+t}}$ show that $X(\R)$ is diffeomorphic to the real locus of (a smooth projective compactification $Y$ of) the blow-up of the affine variety $\{x^2+y^2=(t-t^2), u^2+v^2=(t-t^2)\}\subset\bA^5_{\R}$ at its two singular points. The variety $Y$ is $\R$-rational since the projection to the $\R$-rational surface $\{x^2+y^2=(t-t^2)\}$ is a conic bundle with a rational section $\{u=x,v=y\}$. 
This concludes the proof.
\end{proof}

\bibliographystyle{plain}
\bibliography{luroth}
\end{document}